\documentclass[11pt,reqno]{article}
\usepackage[margin=1.25in]{geometry}
\usepackage{setspace}
\usepackage{amsmath,amsfonts,amsthm,amssymb,enumerate,graphicx}
\usepackage{url,mathtools,tabularx,multirow,esint}

\numberwithin{equation}{section}

\let\div\relax
\DeclareMathOperator\div{div}

\DeclareMathOperator\curl{curl}
\DeclareMathOperator\Id{Id}
\DeclareMathOperator\supp{supp}

\DeclareMathOperator\dist{dist}

\newcommand\dd[0]{\partial}
\newcommand\grad[0]{\nabla}
\newcommand\Td[0]{{\mathbb T^3}}
\newcommand\Rd[0]{{\mathbb R^3}}
\newcommand\Zd[0]{{\mathbb Z^3}}
\newcommand\eq[1]{\begin{align}{#1}\end{align}}
\newcommand\eqn[1]{\begin{align*}{#1}\end{align*}}
\newcommand\N[1]{N_{#1}}
\newcommand\M[1]{M_{#1}}

\theoremstyle{plain}
\newtheorem{thm}{Theorem}[section]
\newtheorem{lem}[thm]{Lemma}

\newtheorem{proposition}[thm]{Proposition}

\theoremstyle{definition}
\newtheorem{define}[thm]{Definition}

\theoremstyle{remark}
\newtheorem{remark}[thm]{Remark}

\title{Non-Uniqueness of Smooth Solutions of the Navier--Stokes Equations from Critical Data}
\date{March 18, 2025}

\author{Matei P. Coiculescu\thanks{Department of Mathematics, Princeton University, Princeton, NJ 08544.\\Email: coiculescu@princeton.edu}\,\, and Stan Palasek\thanks{Department of Mathematics, Princeton University and School of Mathematics, Institute for Advanced Study, Princeton, NJ 08540. Email: palasek@ias.edu}}

\begin{document}

\maketitle

\begin{abstract}
We consider the Cauchy problem for the incompressible Navier--Stokes equations in dimension three and construct initial data in the critical space $BMO^{-1}$ from which there exist two distinct global solutions, both smooth for all $t>0$. One consequence of this construction is the sharpness of the celebrated small data global well-posedness result of Koch and Tataru. This appears to be the first example of non-uniqueness for the Navier--Stokes equations with data at the critical regularity. The proof is based on a non-uniqueness mechanism proposed by the second author in the context of the dyadic Navier--Stokes equations.
\end{abstract}

\section{Introduction}\label{introduction}

Consider the initial value problem for the incompressible Navier--Stokes equations on $\Td$ or $\Rd$:
\begin{equation}\begin{aligned}\label{NSE}
\dd_tu-\nu\Delta u+u\cdot\grad u+\grad p=0\\
\div u=0\\
u(0,x)=U^0(x)
\end{aligned}\end{equation}
where the given initial data $U^0(x)$ is divergence-free, and $u(x,t)$ and $p(x,t)$ are the unknown velocity and pressure fields. By rescaling, we assume without loss of generality that $\nu=1$. Initiated in 1934 by the pioneering work of Leray~\cite{leray1934mouvement}, there is now a well-developed theory of local well-posedness of \eqref{NSE} from initial data in many reasonable classes of functions, as well as global well-posedness from data that is sufficiently small in certain scale-invariant norms; see, for instance, \cite{lemarie2002recent} for an extensive survey. The strongest result in this direction is the following well-known theorem of Koch and Tataru~\cite{koch2001well}:

\begin{thm}\label{KTtheorem}
There exists $\epsilon>0$ such that if $U^0$ is divergence-free with\footnote{See \S\ref{wellandillposednesssubsection} for the definition of $BMO^{-1}$ and further discussion on why this is a natural setting for these questions.} $\|u\|_{BMO^{-1}}<\epsilon$, then there exists a unique global-in-time solution that is regular for $t>0$.
\end{thm}

To this point, it was unknown whether any of the positive results for small data are sharp in the sense that global well-posedness can fail for large initial data. This question was asked, for instance, by Koch and Tataru in \cite{koch2001well} in relation to Theorem~\ref{KTtheorem}. We answer their question in the affirmative:
\begin{thm}\label{maintheorem}
There exists divergence-free initial data $U^0\in BMO^{-1}$ such that the Cauchy problem \eqref{NSE} admits two distinct global solutions
\eqn{
u^{(1)},\,u^{(2)}\in C_{t,x}^\infty((0,\infty)\times\Td)\cap L^\infty([0,\infty);BMO^{-1}(\Td))\cap C^0([0,\infty);\dot W^{-1,p}(\Td))
}
for all $p<\infty$.
\end{thm}

\begin{remark}
The solutions constructed in the proof of Theorem~\ref{maintheorem} belong precisely to the same class as those constructed by Koch and Tataru; the only difference is that they are not perturbative around zero. Indeed, one can verify that our solutions obey the following bounds:
\eqn{
\|u\|_{X_{KT}}&\coloneqq \sup_{t>0}t^\frac12\|u(t)\|_{L^\infty(\Td)}+\sup_{x_0\in\Td}\sup_{R>0}\left(R^{-3}\int_0^{R^2}\int_{B(R)}|u(x,t)|^2dxdt\right)^{1/2}<\infty.
}
In other words, they lie in the natural path space corresponding to solutions of \eqref{NSE} (or even the heat equation) with initial values in $BMO^{-1}$. It was shown in \cite{koch2001well} that the bilinear operator appearing in the mild form of the Navier--Stokes equations is bounded on $X_{KT}$, which allows one to use Kato's method to construct global solutions from small data.
\end{remark}

\begin{remark}
The $BMO^{-1}$ norm is critical with respect the Navier--Stokes scaling symmetry. In other words, if $u(t,x)$ is a global solution of \eqref{NSE} with initial data $U^0$, then
\eq{\label{rescaling}
u_\lambda(x,t)=\lambda u(\lambda x,\lambda^2t),\quad U_\lambda^0(x)=\lambda U^0(\lambda x)
}
is also a solution for any $\lambda>0$, and furthermore
\eqn{
\|u_\lambda\|_{X_{KT}}=\|u\|_{X_{KT}},\quad\|U_\lambda^0\|_{BMO^{-1}}=\|U^0\|_{BMO^{-1}}.
}
Critical spaces are a natural setting for the questions we consider because they are at the threshold of the well-posedness theory; there are many theorems to the effect that certain classes of critical solutions are well-behaved \cite{escauriaza2003l_3,gallagher2016blow,barker2018uniqueness,barker2020localized,chemin1999theoremes,auscher2004stability,miura2005remark,may2010extension,marchand2007remarques,lemarie2007uniqueness,zhang2009uniqueness,hou2024regularity}.
To our knowledge, Theorem~\ref{maintheorem} is the first ill-posedness result in a critical space strong enough to support well-posedness for small data. Moreover, it appears to be the first example of non-uniqueness for the Navier--Stokes equations from data in \emph{any} critical space that does not rely on the presence of an external force. See \S\ref{previousworksection} for a comparison of our solutions to those constructed previously in the literature.
\end{remark}

\begin{remark}
Let us emphasize that the solutions constructed in Theorem~\ref{maintheorem} are not in the Leray--Hopf class due to the fact that the initial velocity is not in $L^2(\Td)$. The construction here uses a simplified version of the mechanism proposed by the second author in \cite{palasek2024non} and this simplification requires the data to have infinite energy. If the full mechanism from \cite{palasek2024non} could be implemented, we expect it would yield solutions of \eqref{NSE} in a similar class as those of Jia and \v Sver\'ak~\cite{jia2014local,jia2015incompressible} (but without exact self-similarity), namely Leray--Hopf solutions belonging to improved critical spaces at least as good as $B^{-1+3/p}_{p,\infty}$ for $p>3$. There does not appear to be any obstruction in principle to realizing the mechanism proposed in \cite{palasek2024non}, but there exists a substantial technical challenge in finding the correct building blocks that embed the behavior of the Obukhov dyadic model (cf. the discussion in Section 1.3 of \cite{tao2016finite}).
\end{remark}

\begin{remark}
The initial data $U^0$ is smooth outside of a measure zero set $\Sigma\subset\Td$. Indeed, it can be written as $U^0=\sum_{k=1}^\infty V_k^0$ where each component $V_k^0$ is smooth on $\Td$ and the supports are monotone, $\supp V_k^0\subset \supp V_{k-1}^0$. Thus $\Sigma\subset\bigcap_{k\geq0}\tilde\Omega_k$, with $\tilde\Omega_k$ as in Definition~\ref{mikadodefinition}. The claim then follows because $|\tilde\Omega_k|\leq 2^{-k}$ from Lemma~\ref{supportlemma}.
\end{remark}

\begin{remark}
For convenience, we carry out the construction on the torus $\Td=\Rd/(2\pi\mathbb Z)^3$. We expect that the argument extends with straightforward modifications to the Navier--Stokes equations posed on $\Rd$ with decay at spatial infinity. In principle, our methods extend to the two-dimensional case as well (which is likewise covered by the small data well-posedness results of~\cite{koch2001well}, see also \cite{germain2006equations}), but there are technical difficulties similar to those faced in the approach to Onsager's conjecture due to the geometry of Mikado flows. We leave open the problem of non-unique solutions with critical data on $\mathbb T^2$ and $\mathbb R^2$.
\end{remark}

After preparation of the manuscript, Cheskidov, Zeng, and Zhang posted a preprint~\cite{cheskidov2025global} in which they use convex integration to construct wild weak solutions of the Navier--Stokes and MHD from \emph{any} $H^\frac12$ initial data. This is an interesting contribution that implies the lack of weak-strong uniqueness of dissipative solutions. We wish to emphasize that the solutions produced in~\cite{cheskidov2025global} do not overlap with Theorem~\ref{maintheorem} at all; while they have the advantage that the data is in $L^2(\Rd)$, they do not appear to possess any further regularity for $t>0$, nor does the energy inequality hold. The present paper is concerned with the question of well-posedness of mild solutions, which is orthogonal to the setting treated in \cite{cheskidov2025global}. This can be seen from the fact that they construct distinct solutions even for \emph{small} data in $H^\frac12(\Rd)$.

\subsection{Well-posedness and Ill-posedness at the Critical Regularity}\label{wellandillposednesssubsection}

We recall the scale of Besov spaces $B^s_{p,q}$ equipped with the norms $\|u\|_{B^s_{p,q}}=\big\|N^s\|P_Nu\|_{L^p}\big\|_{\ell^q_N}$ where $P_N$ are the Littlewood--Paley projections defined in \S\ref{littlewoodpaleysection}. The spaces $B^{-1+3/p}_{p,q}$ are critical for the three-dimensional Navier--Stokes equations for all $p,q\in[1,\infty]$, although they primarily play a role in the negative regularity regime $p>3$. They are related to the other well-known critical spaces\footnote{We focus on this hierarchy of critical spaces for simplicity of the exposition. Certainly there are many others of note: weak $L^3$, Morrey spaces, multiplier spaces, and more that can be found in \cite{lemarie2002recent}.} as follows:
\eqn{
H^{\frac12}\subset L^3\subset B^{-1+3/p}_{p,\infty}\subset BMO^{-1}\subset B^{-1}_{\infty,\infty}
}
where $3<p<\infty$. Global well-posedness from small data in the first three spaces was proved by Fujita--Kato~\cite{fujita1964navier}, Kato~\cite{kato1984strong}, Cannone~\cite{cannone1994ondelettes}, and Planchon~\cite{planchon1998asymptotic}, at which point the belief was that the natural endpoint of the hierarchy is the space $B_{\infty,\infty}^{-1}$ which is so large it contains \emph{all} translation-invariant spaces of tempered distributions at the critical regularity. The insight of Koch and Tataru~\cite{koch2001well} was that $BMO^{-1}$, defined as the distributions with
\eqn{
\|u\|_{BMO^{-1}}\coloneqq \sup_{R>0}\sup_{x_0\in\Rd}\left(\int_0^{R^2}\!\int_{B(x_0,R)}|e^{t\Delta}u|^2dxdt\right)^\frac12<\infty,
}
falls in between $B_{p,\infty}^{-1+3/p}$ and $B_{\infty,\infty}^{-1}$ and is such that the Navier--Stokes bilinear operator
\eqn{
u\mapsto -\int_0^te^{(t-t')\Delta}\mathbb P\div u(t')\otimes u(t')dt'
}
is bounded, allowing one to implement Kato's method for constructing mild solutions from small data. Figure~\ref{table} summarizes the current understanding of well-posedness of \eqref{NSE} in the critical setting and how our main theorem fits into that picture.

\begin{figure}\label{table}\begin{center}

{\small\renewcommand{\arraystretch}{1.7}
\begin{tabular}{r|c|c|c|}
  \multicolumn{1}{l}{} & 
  \multicolumn{1}{c}{$H^{1/2}$, $L^3$, $B_{p,\infty}^{-1+3/p}$} & 
  \multicolumn{1}{c}{$BMO^{-1}$} & 
  \multicolumn{1}{c}{$B^{-1}_{\infty,\infty}$} \\
  \cline{2-4}
  Small data  g.w.p. & \multirow{2}{2.8cm}{\parbox{2.8cm}{\centering Yes; Fujita--Kato, etc. \cite{fujita1964navier,kato1984strong,cannone1994ondelettes}}} & Yes; Koch--Tataru~\cite{koch2001well} & \multirow{3}{2.8cm}{\parbox{2.8cm}{\centering No; Bourgain--Pavlovi\'c~\cite{bourgain2008ill}}} \\
  \cline{3-3}
  Large data l.w.p. &  & \multirow{2}{2.8cm}{\parbox{2.8cm}{\centering No; Theorem~\ref{maintheorem}}} &  \\
  \cline{2-2}
  Large data g.w.p. & Open &  &  \\
  \cline{2-4}
\end{tabular}}

\caption{For initial data in various critical spaces, one can ask whether the Navier--Stokes equations are locally or globally well-posed. In some cases, the answer depends on the size of the data. For the ``open'' case, the answer may depend on the exact space; for instance a negative answer for $B_{p,\infty}^{-1+3/p}$ is suggested by \cite{jia2014local,jia2015incompressible,guillod2023numerical}, while a positive or negative answer in $H^{1/2}$ or $L^3$ would resolve the Clay problem.}
\end{center}
\end{figure}

One can argue that $BMO^{-1}$ is the largest reasonable critical space with this property. Indeed, to make sense of the nonlinearity, even in the sense of distributions, $u$ must be $L^2_{loc}$ in spacetime. This must be true in particular for the zeroth Picard iterate $e^{t\Delta}u_0$, so we need $\|e^{t\Delta}u_0\|_{L_{t,x}^2([0,1]\times B(1))}=C<\infty$. Due to the symmetries of \eqref{NSE}, this bound must be invariant under spatial translation as well as the scaling defined in \eqref{rescaling}. Thus, one is led precisely to the above definition of $BMO^{-1}$, corresponding to the Carleson formulation of $BMO$.

In critical spaces weaker than $BMO^{-1}$ in which Theorem~\ref{KTtheorem} does not apply, there has been progress in showing ill-posedness in the form of norm inflation. The first such result was due to Bourgain and Pavlovi\'c~\cite{bourgain2008ill}, who showed that arbitrarily small data in $B^{-1}_{\infty,\infty}$ can become arbitrarily large in a short time. In other words, the data-to-solution map is not continuous at $0$ in the $B^{-1}_{\infty,\infty}$ topology. Germain~\cite{germain2008second} ruled out the fixed point argument for constructing mild solutions from data in $B^{-1}_{\infty,q}$ for $q>2$. Explicit ill-posedness in this class was proven by Yoneda~\cite{yoneda2010ill} following the ideas of \cite{bourgain2008ill}. Later, Wang~\cite{wang2015ill} proved the surprising result that norm inflation from small data is possible in the spaces $\dot B_{\infty,q}^{-1}$ even for $q\in[1,2]$, which are continuously embedded in $BMO^{-1}$. Thus, although these solutions are globally regular and stay small in the $BMO^{-1}$ norm, they nonetheless become large in the stronger norm $B_{\infty,q}^{-1}$. The fact that the presence of norm inflation is not ``monotone'' in the class of data is interesting and suggests that non-uniqueness (or the lack thereof) may be a better benchmark for the well-posedness of \eqref{NSE}. Let us also mention a slight strengthening of these results in a logarithmically subcritical Besov space due to Cui~\cite{cui2015sharp}.

When one allows a force, recent works by C\'ordoba, Mart\'inez-Zoroa, and Zheng demonstrate that finite-time blow-up is possible for Euler, hypodissipative Navier--Stokes, and other related models. For instance, see \cite{cordoba2023blow} and \cite{cordoba2024finite}. These constructions have a similar spirit to the one in the present work in the sense that they build an explicit multi-scale initial datum whose evolution forward in time can be reasonably well-understood.

\subsection{Known Approaches towards Non-Uniqueness}\label{previousworksection}

Up to this point, there have been two strategies for producing non-unique solutions to \eqref{NSE} and other equations of fluid dynamics. Here we briefly summarize the state of the art.

\subsubsection{Convex Integration}

Inspired by Nash's $C^1$ isometric embedding theorem~\cite{nash1954c}, De Lellis and Sz\'ekelyhidi introduced convex integration to the field of mathematical fluid dynamics in their seminal papers ~\cite{de2009euler,de2013dissipative}. Whereas for the Euler equations these methods produced essentially sharp results (culminating, for instance, in the resolution of Onsager's conjecture~\cite{isett2018proof}), for the Navier--Stokes equations there remain some persistent gaps between the regularity reached by convex integration schemes and the known well-posedness results. This area has been highly active so we refer the reader to \cite{buckmaster2020convex} for a more thorough survey.

The first example of non-unique weak solutions of \eqref{NSE} was constructed in the space $C_tH_x^{\epsilon}$ for $\epsilon>0$ small in the breakthrough work of Buckmaster and Vicol~\cite{buckmaster2019nonuniqueness}. For comparison, the critical space in this scale in dimension three is $C_tH_x^{1/2}$, and Leray--Hopf solutions additionally lie in $L_t^2H_x^1$. Tao~\cite{taoblognavierstokesconvexintegration} observed that one can reach regularity $C_tH_x^{\frac12-\epsilon}$ by taking the dimension $n$ large depending on $\epsilon$, but then the critical Sobolev space is $C_tH_x^{\frac n2-1}$, which is far out of reach. Among the subsequent convex integration constructions following \cite{buckmaster2019nonuniqueness}, for instance \cite{buckmaster2021wild,luo2020non,cheskidov2022sharp,cheskidov20232}, the sharpest in terms of the scaling of the space are those of Cheskidov and Luo: first in \cite{cheskidov2022sharp}, they construct solutions in $L_t^{2-\epsilon}L_x^\infty(\mathbb R^n)$ for $n\geq2$ and all $\epsilon>0$. Later in \cite{cheskidov20232}, they exhibit solutions in the space $C_tL_x^{2-\epsilon}(\mathbb R^2)$, where $\epsilon>0$ can be taken arbitrarily small. For comparison, $L_t^2L_x^\infty$ and $L_t^\infty L_x^n$ are critical in dimension $n$; thus, their solutions can get arbitrarily close to certain critical spaces on the Prodi--Serrin--Ladyzhenskaya scale and possess some additional favorable properties outside of a small set of times.

Unfortunately, convex integration at or above critical regularity appears out of reach. Indeed, the nonlinearity must be at least as large as the Laplacian in order for the error correction to function, so there is no room to lose anything during the iteration. Moreover, convex integration is only known to be applicable in settings where the solution space has flexibility, which is certainly not the case for solutions that immediately regularize like the ones in Theorem~\ref{maintheorem}, nor (one might expect) for Leray--Hopf solutions, which exhibit weak-strong uniqueness. There are additional difficulties related to finding sufficiently intermittent building blocks which present an obstacle to, among other things, proving the sharpness of the full Prodi--Serrin--Ladyzhenskaya scale.

\subsubsection{The Jia--\v Sver\'ak Program}

The other known strategy for demonstrating non-uniqueness was set forth by Jia and \v Sver\'ak in \cite{jia2014local,jia2015incompressible}. They propose to consider \eqref{NSE} in self-similar coordinates $\xi=x/t^{1/2}$, $\tau=\log t$ with the ansatz $u(t,x)=t^{-1/2}U(\xi,\tau)$. Thus the question of non-uniqueness for $u$ at $t=0$ is transferred to $U$ at $\tau=-\infty$, which would be settled if one can find a steady solution $U(\xi,\tau)=\overline U(\xi)$ that is nonlinearly unstable in the self-similar dynamics. This scenario seems plausible in light of compelling numerical evidence by Guillod and \v Sver\'ak~\cite{guillod2023numerical}. Recently, Albritton, Bru\'e, and Colombo~\cite{albritton2022non} used this framework to construct non-unique Leray--Hopf solutions to the forced Navier--Stokes equations by building a profile $\overline U$ out of the unstable vortex constructed by Vishik in \cite{vishik1,vishik2}. The force is necessary for their construction because $\overline U$ does not satisfy the Navier--Stokes in self-similar coordinates. Removing the force appears extremely difficult because the only suspected unstable self-similar profile is the one numerically discovered in \cite{guillod2023numerical}. It follows that one would have to verify spectral properties of the equation linearized around a numerical profile.

\subsection{Strategy of the Proof}\label{strategysubsection}

The non-uniqueness mechanism that leads to Theorem~\ref{maintheorem} is based on the one proposed by the second author in the context of the Obukhov dyadic model in~\cite{palasek2024non}, with certain simplifications to make the construction feasible in the full PDE setting. Let us now sketch the strategy of the proof, neglecting some details for the purpose of the exposition.

\subsubsection{Non-Uniqueness Mechanism}

The initial data\footnote{Let us mention for the reader trying to reconcile this sketch with the proof in the sequel that we shall not in practice work with $V_k^0$; for technical reasons the potentials $\psi_k^0$ satisfying $V_k^0=\curl\curl\psi_k$ are more favorable.} is constructed as a lacunary series $U^0=\sum_{k\geq0}V_k^0$ where each $V_k^0$ is (approximately) localized in Fourier space to a band around $|\xi|=\N{k}$, with $\N{k}$ a rapidly growing sequence of frequency scales. The essence of the construction is the following claim: there exists a choice of the $(V_k^0)_{k\in\mathbb  N}$ such that at each frequency level $k$, there exist two evolutions forward in time that are consistent with \eqref{NSE}, up to a small error:

\begin{itemize}
    \item \emph{The Heat-Dominated Flow}: let $v_k(x,t)$ be the flow of the data under the heat equation, namely
    \begin{equation}\begin{aligned}\label{heatdominated}
    \dd_tv_k-\Delta v_k=l.o.t.\\
    v_k|_{t=0}=V_k^0
    \end{aligned}\end{equation}
    where we add some lower order terms as a matter of convenience. Clearly $\|v_k(t)\|_{L^\infty}$ decays exponentially on the time scale $\N{k}^{-2}$.
    \item \emph{The Inverse Cascade-Dominated Flow}: let $\overline v_k(x,t)$ be the flow of the data under the equation
    \begin{equation}\begin{aligned}\label{cascadedominated}
    \dd_t\overline v_k+P_{\sim\N{k}}\mathbb P\div v_{k+1}\otimes v_{k+1}=0\\
    \overline v_k|_{t=0}=V_k^0
    \end{aligned}\end{equation}
    where $v_{k+1}$ is the ``heat-dominated flow'' emanating from the data $V_{k+1}^0$ as described above, and $P_{\sim \N{k}}$ is a Littlewood--Paley-type projection to a frequency shell around $\N{k}$. Crucially, by a particular choice of $V_{k+1}^0$, we can arrange that $\|\overline v_k(t)\|_{L^\infty}$ decays exponentially to zero on the time scale $\N{k+1}^{-2}$.
\end{itemize}

Taking for granted the assertions made so far (which, indeed, most of the paper will be occupied with justifying), we can write down two distinct \emph{approximate} solutions of \eqref{NSE} with data $U^0=\sum_{k\geq0}V_k^0$:
\eq{\label{v1v2def}
v^{(1)}\coloneqq v_0+\overline v_1+v_2+\overline v_3+\cdots\quad\text{and}\quad v^{(2)}\coloneqq \overline v_0+v_1+\overline v_2+v_3+\cdots.
}
The distinctness is immediate from the fact that $v_k$ and $\overline v_k$ have different decay rates.

Let us give a more heuristic description of how the non-uniqueness arises. Fixing $k\geq1$, one should interpret \eqref{cascadedominated} as asserting that $V_{k+1}^0$ is capable of using its nonlinear self-interaction to ``annihilate'' the mode below that has initial data $V_k^0$, and this annihilation occurs on time scale $\N{k+1}^{-2}$. Likewise, $V_k^0$ is capable of annihilating the mode at $V_{k-1}^0$ on time scale $\N{k}^{-2}$. However, these two events are incompatible because the $k$th mode would vanish on a time scale (i.e., $\N{k+1}^{-2}$) much shorter than the time it needs to act on $v_{k-1}$ (i.e., $\N{k}^{-2}$). Thus, if $V_k^0$ is in fact annihilated by $V_{k+1}^0$ (i.e., its continuation $\overline v_k(t)$ satisfies \eqref{cascadedominated}), then there are no nonlinear dynamics to act on $V_{k-1}^0$, which then evolves as the heat-dominated flow $v_{k-1}(t)$ according to \eqref{heatdominated}.

The conclusion of the discussion above is that these dynamics are only self-consistent if \eqref{heatdominated} and \eqref{cascadedominated} occur at alternating sequences of $k$'s, which justifies the choice \eqref{v1v2def}. The non-uniqueness then is a consequence of the fact that \eqref{heatdominated} can occur at \emph{either} every odd $k$ \emph{or} every even $k$.

\subsubsection{Structure of the Principal Part}

An obvious difficulty is that even if we can arrange \eqref{heatdominated}--\eqref{cascadedominated}, there remains a plethora of nonlinear interactions between the $v_k$'s and $\overline v_k$'s that need to be small for us to claim these as approximate solutions. It turns out that if the data is chosen appropriately, the terms appearing in \eqref{heatdominated}--\eqref{cascadedominated} are the only ones that are non-perturbative. Toward this end, we borrow certain elements from convex integration constructions\footnote{We wish to emphasize that despite having some elements in common with convex integration, the construction here is not in any sense an example of convex integration. Indeed, we only have freedom to choose the data. Outside of the alternative between \eqref{heatdominated} and \eqref{cascadedominated}, for positive times there is no flexibility and the solutions are smooth.} where a central difficulty is also minimizing nonlinear interactions between solution components at different frequencies.

Now we give the rough ideas for how to select data obeying \eqref{heatdominated}--\eqref{cascadedominated} while keeping all other interactions minimized. The initial data $V_k^0(x)$ is (very roughly, and locally) of the form $a_k(x)\theta\sin(\N{k}x\cdot\eta)$ for some fixed $\eta,\theta\in\Zd$ with $\eta\cdot\theta=0$, where $a_k(x)$ is a scalar coefficient that is principally supported at frequencies $|\xi|\ll\N{k}$. Then, according to \eqref{heatdominated}, we should take $v_k(x,t)=a_k(x)\theta\sin(\N{k}x\cdot\eta)\exp(-\N{k}^2|\eta|^2t)$ up to various small errors. Having defined the $(v_k)_{k\in\mathbb N}$ in this way, the requirement \eqref{cascadedominated} becomes
\eqn{
\dd_t\overline v_k=-P_{\sim\N{k}}\mathbb P\div\Big(a_{k+1}^2\theta\otimes\theta\sin^2(\N{k+1}x\cdot\eta)\Big)\exp(-2\N{k+1}^2|\eta|^2t)\\
\overline v_k|_{t=0}=V_k^0.
}
One recalls that $a_{k+1}^2(x)$ is mainly localized to frequencies $\ll \N{k+1}$, while $\sin^2(\N{k+1}x\cdot\eta)$ is supported at frequencies $|\xi|=0$ and $|\xi|=2|\eta|\N{k+1}$. Thus the interaction of $a_{k+1}^2$ with the high-frequency part of $\sin^2$ cannot concentrate significantly in the shell $|\xi|\sim \N{k}$, so only the zeroth mode of $\sin^2$ contributes. This lets us reduce \eqref{cascadedominated} by integrating in time to 
\eqn{
\overline v_k(t)=V_k^0-C\N{k+1}^{-2}P_{\sim\N{k}}\mathbb P\div(a_{k+1}^2\theta\otimes\theta)(1-\exp(-2\N{k+1}^2|\eta|^2t))
}
up to errors, where $C$ is a computable constant. Thus, our claim that $\|\overline v_k(t)\|_{L^\infty}$ decays on time scale $\N{k+1}$ is achievable if $a_{k+1}$ is chosen so that
\eq{\label{simplerequirement}
V_k^0=C\N{k+1}^{-2}P_{\sim\N{k}}\mathbb P\div(a_{k+1}^2\theta\otimes\theta).
}
As stated, this is not quite possible, but the reader familiar with convex integration will recognize the issue is analogous to a problem first encountered by Nash in \cite{nash1954c} that can be solved, for instance, by replacing the simple ansatz $V_k^0=a_k(x)\theta\sin(\N{k}x\cdot\eta)$ with one based on a Mikado flow, i.e., a superposition of spatially disjoint shear flows. These have the advantage that the rank-1 stress $a_{k+1}^2\theta\otimes\theta$ now becomes a linear combination $\sum_ja_{j,k+1}^2\theta_j\otimes\theta_j$. We observe that $V_k^0$ can be constructed to be in the form $\curl\curl\psi_k^0$ which makes it not only divergence free, but also in the form $\div R_k$ for some symmetric tensor $R_k$. Thus $a_{k+1}$ can be chosen to satisfy \eqref{simplerequirement} exactly.

Imposing \eqref{simplerequirement} at every $k\geq0$ creates a recursive dependence between the modes of the initial data: having chosen $V_0^0$, the requirement \eqref{simplerequirement} determines $a_1$, which in turn determines $V_1^0$, which then determines $a_2$, and so on. Based on this, we can estimate the amplitudes of the $V_k^0$. Roughly speaking, $\|V_k^0\|_{L^\infty}\sim\|a_k\|_{L^\infty}$. Then \eqref{simplerequirement} leads to
\eq{\label{recursivesize}
\|V_{k+1}^0\|_{L^\infty}\approx C_1\N{k+1}\left(\frac{\|V_k^0\|_{L^\infty}}{\N{k}}\right)^\frac12.
}
As an immediate consequence, we should expect the $k$th component of the data $V_k^0$ to have amplitude on the order $\N{k}$, in accordance with the critical scaling of \eqref{NSE}. Two important observations are in order. First, \eqref{recursivesize} shows why there is no danger from losing constants during the iteration, which is a fatal issue for convex integration schemes at the critical regularity. Indeed, any large constants (independent of $k$) that appear in constructing $V_{k+1}^0$ from $V_k^0$ are removed by successively taking the $1/2$ power. Second, \eqref{recursivesize} shows why this method cannot produce solutions with small data (which would contradict Theorem~\ref{KTtheorem}). We are free to initiate the process with an arbitrarily small choice of $V_0^0$, but as $k\to\infty$ the amplitude invariably approaches the fixed point of \eqref{recursivesize}, which is presumably large.\footnote{One could use this idea to arrange that the data is not just finite in $BMO^{-1}$, but arbitrarily small in a supercritical space like $B^{-1-\epsilon}_{\infty,\infty}$.}

One may notice at this point that the solutions described above belong to $B^{-1}_{\infty,\infty}$ but not $BMO^{-1}$. The required improvement is now a consequence of improving the spatial support of each $V_k^0$. Recall that Mikado flows are supported on a periodic collection of pipe regions, and there is no difficulty in making the volume of the support smaller than, say, $|\Td|/100$. If one is sufficiently careful\footnote{One has to overcome obstacles related to the nonlocality of certain Fourier multipliers that naturally arise in the construction.}, it is further possible to select $a_{k+1}$ satisfying \eqref{simplerequirement} such that $\supp a_{k+1}\subset\supp V_k^0$. It follows inductively that the support of $V_k^0$ is contained in the intersection of the supports of the Mikado flows that were used to define $v_0^0,v_1^0,\ldots,V_k^0$. Because these flows are at different scales, the volume of this intersection is multiplicative and one can bound $|\supp V_k^0|\leq 2^{-k}|\Td|$, which suffices to improve the regularity to $BMO^{-1}$.

\subsubsection{Construction of the Perturbation}

Finally, let us address the last challenge of the proof which is constructing perturbations $w^{(1)}$ and $w^{(2)}$ that correct the errors, i.e., so that $u^{(i)}=v^{(i)}+w^{(i)}$ are exact solutions of \eqref{NSE} with data $U^0$. The idea is in a similar spirit to the constructions by Chemin and Gallagher~\cite{chemin2006global,chemin2009wellposedness,chemin2010large} of global solutions of \eqref{NSE} from initial data that is large in $BMO^{-1}$ (or even $B^{-1}_{\infty,\infty}$) while satisfying a nonlinear smallness condition.\footnote{Interestingly, these types of constructions have been used to prove uniqueness and weak-strong uniqueness (see for instance~\cite{chemin1999theoremes}) but not, to our knowledge, non-uniqueness.} Suppose $v^{(i)}$ solves \eqref{NSE} with a small error in divergence form,
\eqn{
\dd_tv^{(i)}-\Delta v^{(i)}+\mathbb P\div v^{(i)}\otimes v^{(i)}=\div F^{(i)}\\
v^{(i)}|_{t=0}=U^0.
}
Then $w^{(i)}$ must satisfy the perturbed Navier--Stokes equations
\eqn{
\dd_tw^{(i)}-\Delta w^{(i)}+\mathbb P\div(2v^{(i)}\odot w^{(i)}+w^{(i)}\otimes w^{(i)})=-\div F^{(i)}\\
w^{(i)}|_{t=0}=0.
}
This equation is of the same type faced in~\cite{chemin2006global,chemin2009wellposedness,chemin2010large}. The obvious strategy is to construct solutions via a fixed point argument, but this is complicated by the additional linear terms which are not contractive when $v^{(i)}$ is large. In the papers cited, this difficulty is handled by taking the initial data $U^0$ to be in $H^{1/2}$ (or at least $B^{-1}_{\infty,2}$), leading to $e^{t\Delta}U^0\in L_t^2L_x^\infty$. Unfortunately, the $v^{(i)}$ constructed above are not contained in $L_t^2L_x^\infty$ (nor is that the case for, say, the self-similar solutions posited by Jia and \v Sver\'ak) so those methods break down.

Instead, we proceed as follows. For $i=1,2$, let $S^{(i)}(t,t')$ be the semigroup associated to the the linearized equation for $w^{(i)}$, defined for all $0<t'\leq t$. Then $w^{(i)}$ should obey the mild formulation
\eq{\label{wmildform}
w^{(i)}(t)=\int_0^tS^{(i)}(t,t')(w^{(i)}\otimes w^{(i)}-F^{(i)})dt'.
}
Note that $w^{(i)}$ has zero initial data, so to make sense of this mild formulation, we need only define the semigroup for $t'>0$. Of course, it is also important that the estimates on $S^{(i)}(t,t')$ do not deteriorate too severely as $t'/t\to0$. We show that $S^{(i)}(t,t')$ has estimates slightly worse\footnote{The loss comes from the fact mentioned above that $v^{(i)}$ is not in $L^2(\mathbb R_+;L^\infty)$ or $L^1(\mathbb R_+,t^{-\frac12}dt;L^\infty)$. By making the frequency scales $\N{k}$ more lacunary, we can make these norms blow up as slowly as desired near the initial time, thus minimizing the loss.} than $e^{(t-t')\Delta}\mathbb P\div$. This loss is tolerable because by taking the sequence $\N{k}$ to increase faster than exponentially, the errors $F^{(i)}$ become not just small, but small when measured in a subcritical norm; for instance $\|F^{(i)}(t)\|_{L^\infty}\lesssim t^{-1+\alpha}$ for an $\alpha>0$ (cf.\ the critical power $\alpha=0$). Combining all these elements, we construct the $w^{(1)}$ and $w^{(2)}$ as fixed points of the operator in the mild form \eqref{wmildform} in a small ball around zero measured in a slightly subcritical norm. We conclude by showing that $w^{(i)}\to0$ weakly near the initial time and that $w^{(i)}$ are small enough to not disturb the distinctness of $v^{(1)}$ and $v^{(2)}$.

\subsection{Plan of the Paper}

In \S\ref{preliminarysection}, we state some of the basic Littlewood--Paley theory and related estimates (\S\ref{littlewoodpaleysection}--\ref{functionspacessection}) and introduce some of the key parameters that go into the construction (\S\ref{parametersubsection}). In \S\ref{principalpartsection} we build the data $U^0$ and the leading parts $v^{(i)}$ of the solutions: first we construct the Mikado building blocks for the initial data and prove estimates on the supports (\S\ref{mikadoflowsubsection}); then we run the iterative procedure to construct the initial data and prove estimates (\S\ref{dataconstructionsubsection}); followed by construction of the leading parts of the two distinct solutions forward in time (\S\ref{principalpartconstructionsubsection}). \S\ref{perturbationsection} is concerned with constructing the perturbation $w^{(i)}$ that will correct the principal parts $v^{(i)}$: first we estimate the residuals $F^{(i)}$ that are to be corrected (\S\ref{residualsubsection}); then we construct the semigroup $S^{(i)}$ of the Navier--Stokes linearized around the principal parts $v^{(i)}$ (\S\ref{semigroupsubsection}); and finally we construct the perturbations with a fixed point argument (\S\ref{fixedpointsubsection}). In \S\ref{proofoftheoremsection}, we complete the proof of Theorem~\ref{maintheorem}, proving that the two solutions constructed attain the initial data and are distinct. Appendices~\ref{CIappendix}--\ref{multiplierappendix} give some technical tools used in the proofs.

\subsection*{Acknowledgments}
We are grateful to Camillo De Lellis and Terence Tao for their valuable comments on the manuscript. MPC acknowledges the support of the National Science Foundation in the form of an NSF Graduate Research Fellowship and under Grant No. DMS-2350252. SP acknowledges support from the NSF under Grant No. DMS-1926686.

\section{Preliminaries}\label{preliminarysection}

\subsection{Notation}\label{notationsection}

We use the convention $\mathbb N=\{0,1,2,\ldots\}$. Capital letters such at $N$, $M$, etc.\ are reserved for parameters taking values in the dyadic natural numbers $2^\mathbb N$.

We write $\langle x\rangle\coloneqq (1+|x|^2)^\frac12$ for $x\in\Rd$. For any positive expression $Y$, $O(Y)$ stands in for any quantity $X$ satisfying $X\leq CY$. The expressions $X\lesssim Y$, $X\gtrsim Y$, and $X\sim Y$ take on their usual meanings: $X\leq CY$, $X\geq CY$, and $C^{-1}X\leq Y\leq CX$ respectively. We add subscripts to $O$, $\lesssim$, etc.\ to indicate that the constants can depend on the variables in the subscript. See \S\ref{parametersubsection} for more precise details on the dependencies of the constants.

We write $a\vee b$ and $a\wedge b$ for $\max\{a,b\}$ and $\min\{a,b\}$, respectively.

For $a,b\in\Rd$, define the symmetric tensor product $a\odot b\coloneqq \frac12(a\otimes b+b\otimes a)$. Let $S^{3\times 3}(\mathbb R)$ be the space of real symmetric $3\times 3$ matrices equipped with the Frobenius inner product, which we denote by $A:B \coloneqq A_{ij}B_{ij}$. Here and elsewhere, repeated indices are summed according to the Einstein convention.

We use the following Fourier series convention: if $f\in C^\infty(\Td;\mathbb R)$,
\eqn{
\hat f(\xi)=\fint_\Td f(x)e^{-ix\cdot \xi}dx,\quad f(x)=\sum_{\xi\in\Zd}\hat f(\xi)e^{i\xi\cdot x}.
}
Occasionally we write $\mathcal Ff$ for $\hat f$. Here and elsewhere, $\Td$ denotes the length-$2\pi$ torus $(\mathbb R/2\pi\mathbb Z)^3$ and $\fint_\Td$ is the normalized integral $(2\pi)^{-3}\int_\Td$. 

\subsection{Littlewood--Paley Estimates}\label{littlewoodpaleysection}

Let us define a dyadic partition of unity of $\Rd$,
\eqn{
\sigma(\xi)+\sum_{N\geq1}\pi(N^{-1}\xi)=1
}
for $\sigma\in C_c^\infty(\{\xi\in\Rd:|\xi|<1\})$ and $\pi\in C_c^\infty(\{\xi\in\Rd:\frac12<|\xi|<2\})$, where $N$ takes values in the dyadic natural numbers $2^\mathbb N$. The details are standard and we refer the reader to, for instance, \cite{bahouri2011fourier}. For $f$ a distribution on $\Td$, one can define the Littlewood--Paley projections as Fourier multipliers according to
\eqn{
P_Nf(x)&=\sum_{\xi\in\Zd}\pi(N^{-1}\xi)\hat f(\xi)e^{ix\cdot\xi},\\
P_{\leq N}f(x)&=\sum_{\xi\in\Zd}(\sigma(\xi)+\sum_{M\leq N}\pi(M^{-1}\xi))\hat f(\xi)e^{ix\cdot\xi},
}
with the obvious interpretation for related notation such as $P_{>N}$, etc.

More generally, for $m\in C^\infty(\Rd\setminus0;\mathbb C)$ and $f\in\mathcal S(\Td)$, we define the Fourier multiplier with symbol $m$ by
\eqn{
m(\grad/i)f(x):= \sum_{\xi\in\Zd}m(\xi)\hat f(\xi)e^{ix\cdot\xi}.
}
Important examples for our purpose are constant coefficient differential operators with $m$ a polynomial, the heat propagator with $m(\xi)=e^{-t|\xi|^2}$, and the Leray projection $\mathbb P$ with $m(\xi)=\textrm{Id}-|\xi|^{-2}\xi\otimes\xi$. Recall that these operators are all mutually commuting.

Let us state some basic estimates on the interaction of Littlewood--Paley projection with these multipliers. Note that by applying them component-wise, they extend easily to vector- and matrix-valued multipliers.

\begin{lem}[Bernstein inequalities]\label{bernsteinlemma}
Let $m\in C^\infty(\Rd\setminus0;\mathbb C)$ obeying
\eq{\label{multiplierassumption}
|\grad^im(\xi)|\leq A|\xi|^{\alpha-i}
}
for $i=0,1,\ldots,10$. For all $1\leq p\leq q\leq\infty$ and $f\in L^p(\Td;\mathbb R)$,
\eq{\label{bernsteininequality}
\|P_Nm(\grad/i)f\|_{L^q(\Td)}&\lesssim_{p,q}AN^{\alpha+\frac3p-\frac3q}\|f\|_{L^p(\Td)}.
}
\end{lem}
In particular, one obtains the usual Bernstein inequalities for $P_N\grad^n$.

\begin{lem}[Estimates on the heat propagator]\label{heatlplemma}
Let $f\in L^p(\Td;\mathbb R)$ and $p\in[1,\infty]$. Then for all $t\geq0$,
\eq{\label{heatlittlewoodpaleymultiplierestimate}
\|e^{t\Delta}P_Nf\|_{L^p(\Td)}&\lesssim_p \exp(-N^2t/O(1))\|f\|_{L^p(\Td)}.
}
More generally,
\eq{\label{stokesequationsolutionoperatorestimate}
\|e^{t\Delta}P_Nb(\grad/i)f\|_{L^p(\Td)}&\lesssim_{p,b} N^\alpha\exp(-N^2t/O(1))\|f\|_{L^p(\Td)}
}
where $b\in C^\infty(\Rd\setminus0;\mathbb C)$ is any $\alpha$-homogeneous symbol. In other words, $b(\mu\xi)=\mu^\alpha b(\xi)$ for all $\mu>0$.

Moreover, if $b$ is as above with $\alpha>0$,
\eq{
\|e^{t\Delta}f\|_{L^p(\Td)}&\lesssim_p \exp(-t/O(1))\|f\|_{L^p(\Td)},\label{heatmultiplierestimate}\\
\|e^{t\Delta }b(\grad/i)f\|_{L^p(\Td)}&\lesssim_pt^{-\frac\alpha2}\exp(-t/O(1))\|f\|_{L^p(\Td)}\label{heatandderivativemultiplierestimate}
}
where, in \eqref{heatmultiplierestimate}, one must assume additionally that $\fint_\Td f=0$.
\end{lem}

In \eqref{stokesequationsolutionoperatorestimate}, one should think of $b(\xi)$ as being a component of $\grad$, $\grad\mathbb P$, etc. The proofs are deferred to Appendix~\ref{multiplierappendix}.

\subsection{Function Spaces}\label{functionspacessection}

In the sequel we make use of an equivalent definition of $BMO^{-1}$ rather than the one given in \S\ref{introduction}: $u\in BMO^{-1}$ if $u=\div \varphi$ for some vector field $\varphi\in (BMO)^3$. Recall that $\varphi\in L^1_{loc}$ is in $BMO$ if
\eqn{
\|\varphi\|_{BMO}\coloneqq\sup_Q\fint_Q|\varphi(x)-\varphi_Q|dx<\infty
}
where the supremum runs over cubes $Q\subset\Rd$ and $\varphi_Q=\fint_Q\varphi$. $BMO$ thus defined is obviously only a norm when taken modulo constants, but this issue vanishes for $BMO^{-1}$. These definitions are adapted without issue to functions on $\Td$, considering them as periodic functions on $\Rd$. It is easy to see that averages on large cubes are comparable to averages over a single period. Thus, for instance, the spaces $BMO^{-1}$ and $bmo^{-1}$ from \cite{koch2001well} coincide in the periodic setting.

We also use the homogeneous H\"older--Zygmund spaces, which are the spaces of distributions satisfying
\eqn{
\|u\|_{\mathcal C^s}\coloneqq \sup_{N\in2^\mathbb N}N^s\|P_Nu\|_{L^\infty}<\infty
}
where $s\in\mathbb R$. It is well-known that for $s\in(0,1)$, $\mathcal C^s$ coincides with the standard homogeneous H\"older space $C^s$. More generally, when $s=m+\alpha$ with $m\in\mathbb N$ and $\alpha\in(0,1)$, the space $\mathcal C^{s}$ coincides with $C^{m,\alpha}$. For all $s\in\mathbb R$, $\mathcal C^s$ coincides with the Besov space $\dot B^s_{\infty,\infty}$. In particular, $\mathcal C^{-1}=\dot B^{-1}_{\infty,\infty}$ is a space discussed in \S\ref{introduction} which is strictly weaker than $BMO^{-1}$. We will use the fact that for all $s>0$, $\mathcal{C}^{s}\cap L^\infty$ is a multiplicative algebra. Namely, for $f,g
\in \mathcal{C}^s\cap L^\infty$ we have
$$\|fg\|_{\mathcal{C}^s} \lesssim_s \|f\|_{\mathcal{C}^s}\|g\|_{L^\infty} + \|f\|_{L^\infty}\|g\|_{\mathcal{C}^s}.$$

We remark that $\mathcal C^s$ and $C^s$ defined as above are homogeneous and do not control the function at frequency zero. This is as intended, and typically does not make a difference, since most of the fields we will measure have zero average on $\Td$. For such fields, it is clearly the case that $\|f\|_{L^\infty}\lesssim \|f\|_{\mathcal C^s}$ for all $s>0$.

Let us also record the following heat estimate in the H\"older--Zygmund scale, which is an easy consequence of \eqref{stokesequationsolutionoperatorestimate}: for all $m\geq0$, $s_1,s_2\in\mathbb R$ satisfying $m+s_1-s_2\geq0$, and $f\in\mathcal C^{s_2}$, we have
\eq{\label{besovheatestimate}
\|e^{t\Delta}\grad^mf\|_{\mathcal C^{s_1}}\lesssim_{m,s_1,s_2}t^{-\frac{m+s_1-s_2}2}\|f\|_{\mathcal C^{s_2}}.
}
The estimate holds just as well if $f$ on the left-hand side is replaced by $\mathbb Pf$. Furthermore, one can replace $\mathcal C^0$ by the stronger norm $L^\infty$, with the exception of the endpoint case $\|e^{t\Delta}\mathbb Pf\|_{L^\infty}\lesssim \|f\|_{L^\infty}$ which fails.

\subsection{Parameters}\label{parametersubsection}

We summarize the dependencies of the parameters in the construction and their relation to the various constants that appear in the inequalities.

We introduce a parameter $\alpha\in(0,\frac18)$ that measures the subcriticality of the perturbations $w^{(1)}$ and $w^{(2)}$. The data and solutions we construct are active on two sequences of scales $\M{k}$ and $\N{k}$ which we take to be of the form
\eqn{
\M{k}&=\lfloor A^{b^k}\rfloor,\quad \N{k}=\M{k}\lfloor \M{k}^{\gamma-1}\rfloor,\quad k=0,1,2,\ldots
}
where $b>5\vee\frac1{1-8\alpha}$ and $\frac1{1-4\alpha}<\gamma<\frac1{4\alpha+b^{-1}}$ can be chosen freely. Here $\lfloor x\rfloor$ is the largest integer $\geq x$. It is easy to see that the condition on $b$ makes the condition on $\gamma$ satisfiable. One should think of the length scales $\M{k}^{-1}$ and $\N{k}^{-1}$ as the period of the Mikado flow and the oscillation wavelength, respectively.

We record the following elementary estimate which will be useful:
\eq{\label{parameterinequality}
4\alpha<1-\gamma^{-1}<1-b^{-1}.
}

Here we provide some detail on the dependencies of the parameters and constants. Having fixed $b$, $\gamma$, and $\alpha$ to satisfy the above relations, one can choose $\delta$ smaller than an absolute constant specified in the course of the proof of Lemma~\ref{supportlemma}. Then $\delta_0$ is chosen small depending on $\delta$ to satisfy \eqref{pipevolume}, and depending on an absolute constant to satisfy \eqref{pipesseparated}. The additional parameters chosen throughout Sections~\ref{principalpartsection}--\ref{perturbationsection} (e.g., $\epsilon$, $\beta$, $\beta_1$, $c_1$, $\kappa$) are chosen to lie in explicit intervals depending on $b$, $\alpha$, and $\gamma$. Throughout the proof (with the exception of Lemma~\ref{supportlemma} where $\delta$ is chosen), the implicit constants implied by the notation $\lesssim$, $O(\cdot)$, etc.\ may depend on $\alpha$, $\beta$, $\gamma$, $\delta$, $\delta_0$, $\epsilon$, $\kappa$, $b$, and $c_1$ but \emph{not} $A$. The named constants $C_1,C_2,\ldots$ are chosen large depending on implicit constants and all parameters other than $A$. Finally, throughout the construction and proofs, the scale parameter $A$ in the definition of $\M{k}$ is taken to be as large as needed, depending on all the other parameters.

\section{The Principal Part}\label{principalpartsection}

\subsection{Mikado flow potentials}\label{mikadoflowsubsection}

The starting point for the construction of our initial data is the Mikado pipe flow introduced by Daneri and Sz\'ekelyhidi~\cite{daneri2017non}. Fix a small $\delta>0$ to be determined which will be the total volume of the Mikado flow support in each period of $\Td$. For $j\in\{1,2,\ldots,6\}$, let $\theta_j\in\mathbb Z^3$ as in Lemma~\ref{nash-lemma} and $\eta_j\in\Zd\setminus0$ such that $\theta_j\cdot\eta_j=0$, which will give the orientation of the shear flows. Let $\ell_j$ be the $2\pi$-periodic line $x_j+\theta_j\mathbb R\mod(2\pi\mathbb Z)^3$ for some $x_j\in\Td$ to be specified in Lemma~\ref{mikadodatapropertieslemma}. Note that since each $\theta_j \in \mathbb{Z}^3$, each $\ell_j$ has finite winding number.

We fix the radial profile of the pipes $\varphi\in C_c^\infty ([0,1))$ to be even and satisfy $\varphi(0)=1$. Let $\delta_0$ be the pipe width, chosen small enough to keep the total volume of the pipes below $\delta$:
\eq{\label{pipevolume}
\pi\delta_0^2\,\sum_{j=1}^6|\ell_j|\leq\delta,
}
where $|\ell_j|$ is the length of $\ell_j\cap[0,2\pi]^3$. Then we let 
$$\tilde\varphi_j(x) = \varphi(\delta_0^{-1}d_{\mathbb{T}^3}(x,\ell_j))\mathbf1_{d_{\mathbb{T}^3}(x,\ell_j)\leq\delta_0}$$
be the profile that defines the support of the pipes. Let us emphasize that $\tilde\varphi_j$ should be interpreted as a function on $\Td$, i.e., as a $2\pi$-periodic function on $\mathbb{R}^3$.

For technical reasons related to localization in space, we define the Mikado flows in terms of vector potentials rather than velocity fields directly.

\begin{define}\label{mikadodefinition}
We define the Mikado flow vector potential
\eqn{
\Psi_{j,k}^0(x) &= \N{k}^{-2}\tilde\varphi_{j}(\M{k} x) \sin(\N{k}(x-x_j)\cdot\eta_j)\theta_j
}
as well as the (approximately) heat evolved potential
\eqn{
\Psi_{j,k}(x,t) &= \Psi_{j,k}^0(x)\exp(-|\eta_j|^2\N{k}^2t).
}
Further, we define slight expansions of the regions on which they are supported:
\begin{itemize}
    \item $\Omega_{j,k}$ is the $\delta_0$-neighborhood of $\supp\Psi_{j,k}$,
    \item $\widetilde\Omega_{j,k}$ is the $(2\delta_0)$-neighborhood of $\supp\Psi_{j,k}$, and
    \item $\Omega_{k}$ and $\widetilde\Omega_{k}$ are regions containing the intersection of the pipes up to and including those at scale $k$,
    \eqn{
    \Omega_{k}\coloneqq\bigcap_{m\leq k}\bigcup_{j=1}^6\,\Omega_{j,m},\quad \widetilde\Omega_{k}\coloneqq\bigcap_{m\leq k}\bigcup_{j=1}^6\,\widetilde\Omega_{j,m}.
    }
\end{itemize}
\end{define}

Defined in this way, $\Psi_{j,k}$ solve the heat equation up to a small error with data $\Psi^0_{j,k}$. Let us record several basic properties of the Mikado potentials.

\begin{lem}\label{mikadodatapropertieslemma}
With $\Psi_{j,k}^0$ and $\Psi_{j,k}$ as in Definition~\ref{Omegadefintion}, we have the following:
\begin{enumerate}[1.]
\item $\Psi_{j,k}^0$ and $\Psi_{j,k}$ are smooth, $2\pi$-periodic, and not identically zero. Moreover
\eq{\label{mikadoprofilebound}
\|\grad^m\Psi_{j,k}^0\|_{L^\infty(\Td)}\lesssim_m\N{k}^{-2+m}.
}
\item The vector fields $\Psi_{j,k}^0$ and $\Psi_{j,k}(t)$ at a fixed time solve the steady pressureless Euler equations:
\eq{\label{steadyeuler}
\div (\Psi_{j,k}\otimes \Psi_{j,k})=0,\quad \div \Psi_{j,k}=0.
}
\item Each $\Psi_{j,k}^0$ and $\Psi_{j,k}$ is supported in the periodic pipe region $\{x\in\Td:\dist(\M{k}x,\ell_j)\leq\delta_0\}$. Taking $\delta_0>0$ smaller if necessary, there exist $(x_j)_{j=1,\ldots,6}$ such that the supports satisfy
\eq{\label{pipesseparated}
\dist(\supp \tilde\varphi_{j_1},\supp \tilde\varphi_{j_2})>10\delta_0\quad\forall j_1\neq j_2.
}
In particular, it follows that $\dist(\widetilde\Omega_{j_1,k},\widetilde\Omega_{j_2,k})>5\delta_0M_k^{-1}$.
\item With
\eq{\label{l2normalizedpipes}
A_{j,k}\coloneqq\fint_\Td\tilde\varphi_j^2(\M{k}x)\sin^2(\N{k}(x-x_j)\cdot\eta_j)dx,
}
we have $A_{j,k}\sim\delta$.
\end{enumerate}
\end{lem}

\begin{proof}
Properties 1--3 are elementary consequences of Definition~\ref{Omegadefintion}. We emphasize that the choices of $(x_j)_{j=1}^6$ and $\delta_0$ made to achieve \eqref{pipesseparated} do not depend on $k$. Property 4 follows from an easy application of Lemma~\ref{BVlemma}, using the fact that $\fint_\Td\tilde\varphi_j^2\sim\delta$, and that $\N{k}/\M{k}$ can be made larger than any absolute constant by choosing a sufficiently large $A$.
\end{proof}

Next we prove sharp estimates on the volume of $\widetilde\Omega_k$. A useful tool will be the ``improved H\"older's inequality'' from \cite{modena2018non} following \cite{buckmaster2019nonuniqueness} which we quote in Appendix~\ref{CIappendix} for the reader's convenience. Note that for $X\subset\Td$, the volume $|X|$ is defined in the natural way, namely the volume of $X$ restricted to a single period.

\begin{lem}\label{supportlemma}
We have
\eqn{
|\widetilde\Omega_k|\leq2^{-k}|\Td|.
}
Moreover, there is a constant $C_0>1$ such that for all cubes $Q\subset\Rd$,
\eq{\label{cubeintersectionvolumeestimate}
|Q\cap\widetilde\Omega_k|\leq 2^{-(k-k(Q))}|Q|\quad\forall k\geq k(Q)
}
where we define $k(Q)\coloneqq \inf\{k\in\mathbb N:\M{k}\geq C_0\ell(Q)^{-1}\}$.
\end{lem}

\begin{proof}
We emphasize that in this proof, the constants (implicit and otherwise) do not depend on $\delta$ or $\delta_0$. Recall the functions $\tilde\varphi_j(\M{k}\cdot)$ that define the supports of the $\Psi_{j,k}$. Let $\phi\in C^\infty(\Td;[0,1])$ be such that
\begin{itemize}
    \item $\phi\equiv1$ on $\bigcup_j\supp \tilde\varphi_j(x)$,
    \item $\phi\equiv0$ outside of the $\delta_0$-neighborhood of $\bigcup_j\supp \tilde\varphi_j(x)$, and
    \item $|\grad\phi|\leq C'\delta_0^{-1}$ for an absolute constant $C'$.
\end{itemize}
Clearly, by \eqref{pipevolume}, we have
\eq{\label{phiinl1}
\|\phi\|_{L^1}\leq\sum_j|\supp \tilde\varphi_j|\leq4\delta.
}
Then we define $\phi_k(x)=\phi(\M{k}x)$ leading to the pointwise inequality $\mathbf1_{\widetilde\Omega_k}\leq\phi_0\phi_1\cdots\phi_k\leq\mathbf1_{\Td}$. The lemma will follow from the stronger claim
\eqn{
\|\phi_0\phi_1\cdots\phi_k\|_{L^1}\leq(8\delta)^{k}|\Td|,
}
and it suffices to take, say, $\delta=\frac1{20}$. We proceed by induction. There is nothing to prove for $k=0$. For $k\geq1$, we apply Lemma~\ref{BVlemma} with $f=\phi_0\cdots\phi_{k-1}$, $g=\phi$, and $\lambda=\M{k}$ to obtain
\eqn{
\|\phi_0\phi_1\cdots\phi_k\|_{L^1}&\leq\|\phi_0\phi_1\cdots\phi_{k-1}\|_{L^1}\|\phi\|_{L^1}+O(\M{k}^{-1}\|\phi_0\phi_1\cdots\phi_{k-1}\|_{C^1}\|\phi\|_{L^1})\\
&\leq 4\delta(8\delta)^{k-1}|\Td|+O(M_k^{-1}\sum_{j=0}^{k-1}\|\phi_j\|_{C^1}),
}
in the last line using the inductive hypothesis and \eqref{phiinl1}. Note that the constants implied by $O()$ are absolute, coming from Lemma~\ref{BVlemma}. The first term is precisely $\frac12(8\delta)^k|\mathbb T^3|$, so it remains to show the second has the same upper bound. We compute
\eqn{
\sum_{j=0}^{k-1}\|\phi_j\|_{C^1}\leq \sum_{j=0}^{k-1}C'M_j\leq C'kM_{k-1}
}
so it suffices to show
\eq{\label{claim}
O(C'kM_{k-1}M_k^{-1})\leq\frac12(8\delta)^k|\mathbb T^3|.
}
One finds by differentiation that $k\mapsto k(8\delta)^{-k}\M{k-1}\M{k}^{-1}$ is decreasing for $k\geq1$ as long as $1+\log(8\delta)^{-1}-(b-1)(\log A)(\log b)<0$, which can be arranged by taking $A$ sufficiently large depending on $b$ and $\delta$. Thus \eqref{claim} reduces to $O(C'M_0M_1^{-1})\leq\delta|\mathbb T^3|$ which indeed holds, using that $M_0M_1^{-1}\leq2A^{1-b}$ and again taking $A$ sufficiently large depending on $\delta$, $b$, and the absolute constants that appeared.

Next, \eqref{cubeintersectionvolumeestimate} can be proven analogously, letting $f$ be the product of $\phi_0\cdots\phi_{k-1}$ and a bump function localized around $Q$.
\end{proof}

Now, with $\delta$ and $\delta_0$ fixed, the constants in the remainder of the paper are allowed to depend on them, and we shall not make this dependence explicit.

\subsection{Construction of the Data}\label{dataconstructionsubsection}

Having defined suitable potentials for the Mikado flows, we proceed to fix the remaining elements that go into the construction of the initial data.

\begin{define}\label{Omegadefintion}
We define:
\begin{itemize}
\item The differential operator that takes vector fields to fields of symmetric rank-2 tensors,
\eqn{
\mathcal Df\coloneqq -\grad f-\grad f^T+2(\div f)\Id;
}
\item A cutoff function $\chi_k$ which is identically $1$ in $\Omega_{k-1}$, identically $0$ outside of $\widetilde\Omega_{k-1}$, and obeys
\eq{\label{cutoffbounds}
\|\grad^m\chi_k\|_{L^\infty(\Td)}&\lesssim_m \M{k-1}^m;
}
\item A standard mollifier $\phi_k$ at length scale $\N{k+1}^{-\frac13}\N{k}^{-\frac23}$, supported in $B(0,1)$ with $\int_\Td\phi_k=1$;
\item The Nash operators $(\Gamma_j)_{j=1}^6$ which obey
\eqn{M=\sum_{j=1}^6\Gamma_j(M)^2\theta_j\otimes\theta_j}
for all $M\in S^{3\times3}$ within a radius $c_0=\frac1{1000}$ of the identity; see Lemma~\ref{nash-lemma} for details.
\end{itemize}

\end{define}
One readily sees that $\mathcal D$ has the property $\div\mathcal Df=\curl\curl f$; see Remark~\ref{curlcurlremark} for further explanation of this choice.

Finally we can build the velocity potential for the initial data. The construction is carried out inductively on scales.

\begin{define}\label{dataiteratedefinition}
Let $\psi_0^0\coloneqq \N{0}\phi_0*\Psi_{1,0}^0$. Given any $\psi_{k-1}^0$ and $k\geq1$, we inductively define
\eq{\label{psi0kdefinition}
\psi_k^0:= \phi_k*\sum_ja_{j,k}(x)\Psi^0_{j,k}(x)
}
where
\eqn{
a_{j,k}(x)=\N{k}\left(\frac{2\|\mathcal D\psi^0_{k-1} \|_{L^\infty}}{c_0|\eta_j|^2A_{j,k}}\right)^\frac12\chi_k(x)\Gamma_j\big(\Id+c_0\frac{\mathcal D\psi^0_{k-1}(x) }{\|\mathcal D\psi^0_{k-1} \|_{L^\infty}}\big)
}
with $\mathcal D$, $\chi_k$, $\phi_k$, and $\Gamma_j$ as in Definition~\ref{Omegadefintion}.

When $k=0$, \eqref{psi0kdefinition} holds just as well upon setting $a_{j,0}=\N{0}$ when $j=1$ and $a_{j,0}=0$ otherwise.
\end{define}
It is important to point out that $a_{j,k}$ is well-defined because the tensor at which we evaluate $\Gamma_j$ lies in the domain as specified in Appendix~\ref{CIappendix}. Here and elsewhere, the notation $\sum_j$ denotes a sum over $j\in\{1,2,\ldots,6\}$.

Let us comment on the inclusion of the mollifier in \eqref{psi0kdefinition}. Because $\mathcal D\psi_{k-1}^0$ appears in the definition of $\psi_k^0$ (via $a_{j,k}$), there is a potential loss-of-derivatives problem in the proofs of Lemma~\ref{principalpartiterativesteplemma} and Proposition~\ref{inductivepsiboundsproposition}. Similar to many convex integration constructions, mollification at a sequence of length scales shrinking to zero plays a Nash--Moser-type role in the iteration.

We also point out that the coefficients $a_{j,k}$ have amplitude $\sim N_k$, as the factors besides the $N_k$ are of unit size. Furthermore, the frequency support of $a_{j,k}$ is essentially supported (neglecting a small error) at and below $|\xi|\sim M_k$. Thus, $a_{j,k}$ is effectively constant compared to $\Psi_{j,k}^0$ which oscillates near the higher frequency $N_k$.

From Definition~\ref{dataiteratedefinition}, one can extract the following inductive estimates.

\begin{lem}\label{principalpartiterativesteplemma}
There is an absolute constant $C_1>0$ such that
\eq{\label{iterativeboundsI}
\|\psi_k^0 \|_{L^\infty}&\leq C_1\N{k}^{-1}\|\mathcal D\psi^0_{k-1} \|_{L^\infty}^\frac12
}
and
\eq{\label{iterativeboundsII}
\|\mathcal D\psi_k^0\|_{L^\infty}\leq \frac12C_1\|\mathcal D\psi^0_{k-1}\|_{L^\infty}^\frac12\left(1+\N{k}^{-1}\frac{\|\grad\mathcal D\psi^0_{k-1}\|_{L^\infty}}{\|\mathcal D\psi^0_{k-1}\|_{L^\infty}}\right)
}
for all $k\geq1$.
\end{lem}

\begin{proof}
We begin by estimating the coefficient functions $a_{j,k}$. First, by \eqref{Gammabound} and part 4 of Lemma~\ref{mikadodatapropertieslemma}, one has
\eq{\label{preliminaryabound}
\|a_{j,k}\|_{L^\infty} \lesssim \N{k} \|\mathcal{D}\psi_{k-1}^0\|_{L^\infty}^{1/2}.
}
Then, for the gradient,
\eqn{
\|\nabla a_{j,k}\|_{L^\infty} \lesssim \N{k}\|\mathcal{D}\psi_{k-1}^0\|_{L^\infty}^{1/2} \left( \|\nabla \chi_k\|_{L^\infty} + \|\nabla(\Gamma_j\big(\Id+c_0\frac{\mathcal D\psi^0_{k-1}(x) }{\|\mathcal D\psi^0_{k-1} \|_{L^\infty}}\big))\|_{L^\infty}     \right).
}
The first term is easily estimated with \eqref{cutoffbounds}. For the second we have
$$|\nabla(\Gamma_j\big(\Id+c_0\frac{\mathcal D\psi^0_{k-1}(x) }{\|\mathcal D\psi^0_{k-1} )\|_{L^\infty}}\big))|\lesssim \frac{\|\nabla \mathcal{D}\psi_{k-1}^0\|_{L^\infty}}{\|\mathcal{D}\psi_{k-1}^0\|_{L^\infty}}$$
using \eqref{Gammabound}. We conclude
\eq{\label{gradapreliminarybound}
\|\grad a_{j,k}\|_{L^\infty}&\lesssim \N{k} \|\mathcal D\psi_{k-1}^0\|_{L^\infty}^\frac12\left(\M{k}+\frac{\|\nabla \mathcal{D}\psi_{k-1}^0\|_{L^\infty}}{\|\mathcal{D}\psi_{k-1}^0\|_{L^\infty}}\right).
}

Now we can proceed toward proving \eqref{iterativeboundsI}--\eqref{iterativeboundsII}. From the definition of $\psi_k^0$, the $L^\infty$ boundedness of the mollifier $\phi_k*$, \eqref{preliminaryabound}, and \eqref{mikadoprofilebound}, we have
$$\|\psi_k^0\|_{L^\infty} \leq \sum_j(\|a_{j,k}(x)\|_{L^\infty} \|\Psi_{j,k}^0\|_{L^\infty})\lesssim \N{k}^{-1}\|\mathcal D\psi_{k-1}^0\|_{L^\infty}^{1/2}.$$ Choosing $C_1$ large enough to absorb the constants, we immediately have \eqref{iterativeboundsI}. Towards \eqref{iterativeboundsII}, recall that $\mathcal{D}$ is a first-order differential operator so we may apply the Leibnitz rule to find
$$\|\mathcal D \psi_k^0 \|_{L^\infty} \lesssim  \sum_j(\| a_{j,k}\|_{L^\infty}\|\nabla\Psi_{j,k}^0\|_{L^\infty} + \|\nabla a_{j,k}\|_{L^\infty}\|\Psi_{j,k}^0\|_{L^\infty}).$$
Then \eqref{mikadoprofilebound}, \eqref{preliminaryabound}, and \eqref{gradapreliminarybound} immediately yield
\eqn{
\|\mathcal D \psi_k^0 \|_{L^\infty}&\lesssim \|\mathcal{D}\psi_{k-1}^0\|_{L^\infty}^{\frac12}+\N{k}^{-1}\|\mathcal D\psi_{k-1}^0\|_{L^\infty}^\frac12\left(\M{k}+\frac{\|\nabla \mathcal{D}\psi_{k-1}^0\|_{L^\infty}}{\|\mathcal{D}\psi_{k-1}^0\|_{L^\infty}}\right).
}
Taking $C_1$ larger if needed and using the fact that $\M{k}\lesssim \N{k}$, we arrive at \eqref{iterativeboundsII}.
\end{proof}

Iteratively combining the estimates from Lemma~\ref{principalpartiterativesteplemma} yields the following. Note that lower bounds will be crucial, both because $\|\mathcal D\psi_{k-1}^0\|_{L^\infty}$ appears in the denominator in Definition~\ref{dataiteratedefinition}, and because establishing non-uniqueness will require precise asymptotics.

\begin{proposition}\label{inductivepsiboundsproposition}
With $C_2$ an absolute constant, we have
\eq{\label{psidataupperandlowerbounds}
C_2^{-2}(C_2^2\|\mathcal D\psi^0_0\|_{L^\infty})^{2^{-k}}\leq\|\mathcal D\psi_k^0\|_{L^\infty}\leq C_2^{2}(C_2^{2}\|\mathcal D\psi^0_0\|_{L^\infty})^{2^{-k}}
}
for all $k\geq0$. Furthermore
\eq{\label{psidatabounds}
\|\grad^m\psi^0_k\|_{L^\infty}&\lesssim_m \N{k}^{-1+m}
}
for all $k,m\geq0$.
\end{proposition}

\begin{proof}
First, we establish \eqref{psidataupperandlowerbounds} by induction. Suppose the claim holds for $\psi^0_m$ for all $m=1,2,\ldots,k-1$. We begin with the lower bound in \eqref{psidataupperandlowerbounds} and decompose $\mathcal D\psi_k^0=I+II+III+IV$ depending on where the derivative falls: $I$ when it falls on the sine within $\Psi_{j,k}^0$; $II$ when it falls on $\tilde\varphi_j$ within $\Psi_{j,k}^0$; $III$ when it falls on $\chi_k$ within $a_{j,k}$; and $IV$ when it falls on $\Gamma_j$ within $a_{j,k}$. We start with the dominant contribution $I$, and further decompose to extract the unmollified part:
\eqn{
I=I_1+I_2;\quad I_1&=\N{k}^{-2}\sum_ja_{j,k}\tilde\varphi_j(\M{k}x)\mathcal D(\sin(\N{k}(x-x_j)\cdot\eta_j)\theta_j)\\
I_2&=\varphi_k*I_1-I_1.
}
We compute
\eqn{
I_1&=-2\N{k}^{-1}\sum_ja_{j,k}\tilde\varphi_j(\M{k}x)(\theta_j\odot\eta_j)\cos(\N{k}(x-x_j)\cdot\eta_j).
}
Note that the terms in the sum have disjoint supports $\widetilde\Omega_{j,k}$, so let us 
restrict to a particular one. Recall that within $\Omega_{j,k}$, the cutoff $\chi_k$ is identically $1$. Thus
\begin{equation}\begin{aligned}\label{I1exactformula}
I_1\big|_{\,\Omega_{j,k}}=-2^\frac32|\eta_j|(c_0^{-1}A_{j,k}^{-1}\|\mathcal D\psi^0_{k-1}\|_{L^\infty})^\frac12\Gamma_j(\Id+c_0\frac{\mathcal D\psi^0_{k-1}}{\|\mathcal D\psi^0_{k-1}\|_{L^\infty}})\\
\quad\times\tilde\varphi_j(\M{k}x)(\theta_j\odot\eta_j)\cos(\N{k}(x-x_j)\cdot\eta_j).
\end{aligned}\end{equation}
Recall from \eqref{Gammabound} that $\Gamma_j$ is bounded from below; meanwhile the $\eta_j$ are fixed vectors specified to be orthogonal to $\theta_j$ from Lemma~\ref{nash-lemma}. For $x=x_j$ on the line $\ell_j$, both the cosine and $\tilde{\varphi}$ factors are exactly $1$. We conclude that
\eqn{
\|I_1\|_{L^\infty}\gtrsim\|\mathcal D\psi^0_{k-1}\|_{L^\infty}^\frac12\geq C_2^{-1}(C_2^2\|\mathcal D\psi^0_0\|_{L^\infty})^{2^{-k}}
}
by the induction hypothesis. With a large enough choice of $C_2$, this implies $\|I_1\|_{L^\infty} \geq 100C_2^{-2}(C_2^2\|\mathcal D\psi^0_0\|_{L^\infty})^{2^{-k}}$. It remains to show that the other terms are uniformly much smaller.

To estimate $I_2$, we need to bound the gradient of $I_1$:
\eqn{
\|\grad I_1\|_{L^\infty}&\lesssim \frac{\|\grad\mathcal D\psi^0_{k-1}\|_\infty}{\|\mathcal D\psi^0_{k-1}\|_{L^\infty}^{1/2}}+\|\mathcal D\psi^0_{k-1}\|_{L^\infty}^{1/2}\N{k}.
}
The numerator in the first term is $O(\N{k}^{\frac23}\N{k-1}^{\frac13})$, using the mollifier to absorb the derivatives. The denominator is $1/O(1)$ by the inductive hypothesis. Similarly the second term is $O(\N{k})$. Thus, by a standard mollifier estimate,
\eqn{
\|I_2\|_{L^\infty}&\lesssim \N{k+1}^{-\frac13}\N{k}^{-\frac23}\|\grad I_1\|_{L^\infty}\lesssim (\N{k}/\N{k+1})^\frac13=A^{-\frac13\gamma(b-1)b^k}.
}

Next, consider $II$. Recall that $\tilde\varphi_j$ is a fixed smooth function, so clearly $\grad(\tilde\varphi_j(\M{k}x))$ is $O(\M{k})$. From this and the inductive hypothesis, we have
$$\|II\|_{L^\infty} \lesssim \N{k}^{-1}\M{k}\|\mathcal{D}\psi_{k-1}^0\|_{L^\infty}^{1/2}\lesssim_{C_2,\psi_0^0} \N{k}^{-1}\M{k}=A^{-(\gamma-1)b^k}.$$
The argument for $III$ is identical: when the derivative lands on $\chi_k$, it yields a factor of $\M{k-1}$ due to \eqref{cutoffbounds}, which is even better than the factor $\M{k}$ that appeared in $II$. 

We finish by estimating $IV$. When $\mathcal D$ falls on $\Gamma_j$, the chain rule produces a factor $\grad\mathcal D\psi_{k-1}^0$. As in the analysis of $\grad I_1$ above, one places derivatives on the mollifier to obtain $\|\grad\mathcal D\psi_{k-1}^0\|_{L^\infty}\lesssim \N{k}^\frac23\N{k-1}^{\frac13}$ which yields
$$\|IV\|_{L^\infty} \lesssim (\N{k-1}/\N{k})^\frac13=A^{-\gamma(1-b^{-1})b^k}.$$

To summarize, we have shown
\eqn{
\|\mathcal D\psi_k^0\|_{L^\infty}&\geq\|I_1\|_{L^\infty}-\|I_2\|_{L^\infty}-\|II\|_{L^\infty}-\|III\|_{L^\infty}-\|IV\|_{L^\infty}\\
&\geq100C_2^{-2}(C_2^2\|\mathcal D\psi_0^0\|_{L^\infty})^{2^{-k}}-O_{C_2,\psi_0^0}(A^{-c(b,\gamma)b^k})
}
for some $c(b,\gamma)>0$. Notice that the first term stays bounded from below as $k\to\infty$. Thus, choosing $A$ sufficiently large depending on $\psi_0^0$, $C_2$, $b$, and $\gamma$, the second term becomes negligible and we arrive at the desired lower bound.

To prove the upper bound in \eqref{psidataupperandlowerbounds}, the same upper bounds on $I_2$, $II$, $III$, and $IV$ still suffice. One need only recognize that the estimate on $I_1$ is reversible and one can proceed from \eqref{I1exactformula} to obtain $\|I_1\|_{L^\infty}\lesssim C_2(C_2^{2}\|\mathcal D\psi_0^0\|_{L^\infty})^{2^{-k}}$.

To see \eqref{psidatabounds}, one proceeds directly by differentiating the definition. When all the derivatives fall onto the sine, the upper bound is as claimed (analogously to $I_1$). When derivatives fall onto any other part, the upper bound is strictly smaller, which can be easily seen by arguing as above for $II$--$IV$. For instance, when derivatives fall onto $\Gamma_j$, the resulting higher derivatives of $\psi_{k-1}^0$ can be adequately estimated using the mollifier and the inductive hypothesis \eqref{psidataupperandlowerbounds} as was done above for $\grad\mathcal D\psi_{k-1}^0$.
\end{proof}

The velocity potentials $\psi_k^0$ will be the building blocks for the initial velocity $U^0$. The following proposition paves the way to properly defining $U^0$ as a distribution and proving that it lies in $BMO^{-1}$.

\begin{proposition}\label{dataregularityproposition}
With $\psi_k^0$ as in Definition~\ref{dataiteratedefinition}, the sum $\sum_{k=0}^\infty\curl\psi_k^0$ converges in $L^1(\Td)$. Furthermore,
\eqn{
\sum_{k=0}^\infty\curl\psi_k^0\in BMO.
}
\end{proposition}

\begin{proof}
For brevity we write $\zeta_k=\curl\psi_k^0$. By \eqref{psidatabounds}, we have $\|\zeta_k\|_{L^\infty}=O(1)$. It follows from this and Lemma \ref{supportlemma} that $\|\zeta_k\|_{L^1}\lesssim|\supp{\zeta_k}| \lesssim 2^{-k}|\mathbb{T}^3|$. Convergence in $L^1$ then follows because of the absolute convergence of the series
$$\sum_{k} \|\zeta_k\|_{L^1(\mathbb{T}^3)}<\infty.$$

To show $\zeta=\sum_k\zeta_k$ lies in $BMO$, we recall the definition from \S\ref{functionspacessection},
\eqn{
\|\zeta\|_{BMO}=\sup_{Q\subset\Rd}\fint_Q|\zeta-\zeta_Q|.
}
By periodicity it is easy to see that it suffices to consider cubes $Q\subset\Rd$ with side length $\ell(Q)\in(0,\M{0}^{-1})$, say. Fix any such $Q$ and recall $k(Q)=\inf\{k\in\mathbb N:\M{k}\geq C_0\ell(Q)^{-1}\}$ from Lemma~\ref{supportlemma}. We estimate
\eqn{
\fint_Q|\zeta-\zeta_Q|&\leq 2\fint_Q\sum_{k\geq k(Q)-1}|\zeta_k|+\sum_{k\leq k(Q)-2}\|\zeta_k-(\zeta_k)_Q\|_{L^\infty(Q)}=I+II.
}
To bound $I$, observe that the amplitude of $\zeta_k$ is $O(1)$ from \eqref{psidatabounds}, while the volume of the support can be estimated with \eqref{cubeintersectionvolumeestimate} in Lemma~\ref{supportlemma}. We arrive at
\eqn{
2\fint_Q\sum_{k\geq k(Q)-1}|\zeta_k|&\lesssim \sum_{k\geq k(Q)-1}\frac{|Q\cap\widetilde\Omega_k|}{|Q|}\leq\sum_{k\geq k(Q)-1}2^{-(k-k(Q))}\lesssim 1.
}
Now consider $II$. For any $x\in Q$ and $k\leq k(Q)-2$,
\eqn{
|\zeta_k(x)-(\zeta_k)_Q|&=|\fint_Q(\zeta_k(x)-\zeta_k(y))dy|\\
&\leq\fint_Q\|\grad \zeta_k\|_{L^\infty}\ell(Q)dy\\
&\lesssim \N{k}\ell(Q)
}
using \eqref{psidatabounds}. Summing over $k\leq k(Q)-2$, we obtain the upper bound $O(\N{k(Q)-2}\ell(Q))$. By definition of $k(Q)$, in particular the minimality, we have $\M{k(Q)-1}<C_0\ell(Q)^{-1}$ and conclude
\eqn{
II\lesssim \N{k(Q)-2}\ell(Q)\lesssim \N{k(Q)-2}\M{k(Q)-1}^{-1}=\N{k(Q)-2}^{1-b/\gamma}
}
which is bounded, recalling from \S\ref{parametersubsection} that $\gamma<b$.
\end{proof}

Finally we may define the velocity initial data.

\begin{define}\label{initialdatadefinition}
With $\psi_k^0$ as in Definition~\ref{dataiteratedefinition}, we define the data
\eq{\label{vdatadefinition}
U^0=\curl\sum_{k=0}^\infty \curl \psi^0_k.
}
\end{define}

This is well-defined and in $BMO^{-1}$ because of Proposition~\ref{dataregularityproposition}, with the outer $\curl$ taken in the distributional sense.

\subsection{Construction of the Principal Parts of the Solutions}\label{principalpartconstructionsubsection}

The initial data constructed in Definition~\ref{initialdatadefinition} has the essential property that at each frequency level $k$, there are two different continuations consistent with the Navier--Stokes equations up to a small error: one with dissipation dominating and another with the inverse cascade dominating. (See \S\ref{strategysubsection} for further discussion.) We shall call these two alternatives $v_k$ and $\overline v_k$ respectively. By interposing them in different ways we can construct two approximate Navier--Stokes solutions $v^{(1)}$ and $v^{(2)}$ with data $U^0$.

\begin{define}\label{principalpartdefinition}
Recalling the definitions from \S\ref{principalpartsection}, we define for $k\geq0$ the heat-dominated evolution
\eqn{
v_k(t,x)=\phi_k*\curl\curl\sum_ja_{j,k}(x)\Psi_{j,k}(x,t)
}
as well as the inverse cascade-dominated evolution
\eqn{
\overline v_k(t,x)=\frac12\N{k+1}^{-2}\mathbb P\div\sum_j A_{j,k+1}|\eta_j|^2e^{-2|\eta_j|^2 \N{k+1}^2 t}a_{j,k+1}^2(x)\theta_j\otimes\theta_j.
}
We occasionally write $\psi_k\coloneqq\phi_k*\sum_ja_{j,k}\Psi_{j,k}$, which are the corresponding continuations forward in time from data $\psi_k^0$. Then we define the pair of vector fields
\eqn{
    v^{(1)}=\sum_{k\geq0\text{ even}}v_k+\sum_{k\geq0\text{ odd}}\overline v_k
}
and
\eqn{
    v^{(2)}=\sum_{k\geq0\text{ odd}}v_k+\sum_{k\geq0\text{ even}}\overline v_k.
}
\end{define}

\begin{remark}\label{curlcurlremark}
The operator $\curl\curl=-\Delta+\grad\div$ has the advantage that $v_k=\curl\curl\psi_k$ is both divergence-free and of the form $v_k=\div\mathcal D\psi_k$ where $\mathcal D\psi_k$ is a symmetric tensor. The latter property is essential for our eventual use of Lemma~\ref{nash-lemma}.
\end{remark}

\begin{remark}\label{dataremark}
In \S\ref{proofoftheoremsection}, we give a precise proof that the two solutions constructed have identical data. Nonetheless, it is enlightening to observe at this point that this is true formally for the principal parts $v^{(1)}$ and $v^{(2)}$. Indeed, we compute
\eqn{
\overline v_k(0,x)&=\frac12\N{k+1}^{-2}\mathbb P\div \sum_j A_{j,k+1}|\eta_j|^2a_{j,k+1}^2(x)\theta_j\otimes\theta_j\\
&=c_0^{-1}\|\mathcal D\psi_{k}^0\|_{L^\infty}\mathbb P\div\Big(\chi_{k+1}\sum_j\Gamma_j^2(\Id+c_0\frac{\mathcal D\psi_k^0}{\|\mathcal D\psi_k^0\|_{L^\infty}})\theta_j\otimes\theta_j\Big)\\
&=c_0^{-1}\|\mathcal D\psi_{k}^0\|_{L^\infty}\mathbb P\div(\chi_{k+1}(\Id+c_0\frac{\mathcal D\psi_k^0}{\|\mathcal D\psi_k^0\|_{L^\infty}}))\\
&=\mathbb P\div(\chi_{k+1}\mathcal D\psi_k^0)\\
&=\curl\curl\psi_k^0
}
where we have used Lemmas~\ref{nash-lemma} and \ref{Rlemma}, the fact that $\mathbb P\div$ annihilates scalar function multiples of the identity, and that $\chi_{k+1}$ is identically equal to $1$ on $\supp\mathcal D\psi_k^0$. Meanwhile, the fact that $v_k(0,x)=\curl\curl\psi_k^0$ is immediate from the construction in Definitions~\ref{dataiteratedefinition} and \ref{principalpartdefinition}. It follows that formally we have the equality
\eqn{
v^{(i)}|_{t=0}=\sum_{k\geq0}\curl\curl\psi^0_k=U^0
}
for both $i=1$ and $2$.

\end{remark}

Later in \S\ref{perturbationsection}, we shall see that $v^{(1)}$ and $v^{(2)}$ solve the Navier--Stokes up to a small subcritical error and that the error can be corrected by a small perturbation. Before doing so, we need some basic estimates on their sizes.

\begin{proposition}\label{vestimatesproposition}
The two vector fields $v^{(1)}$ and $v^{(2)}$ are divergence-free and obey the estimates
\eq{\label{vweightedlinftybound}
\|\grad^mv^{(i)}(t)\|_{L^\infty(\Td)}\lesssim_mt^{-\frac {(m+1)}2}e^{-\N{0}^2 t/O_m(1)}
}
for all $m\geq0$ and $t>0$. Moreover
\begin{equation}\label{vl2intimelinftyinspacebound}
\|v^{(i)}\|^2_{L^2([t_1,t_2],dt;L^\infty)}+\|v^{(i)}\|_{L^1([t_1,t_2],t^{-\frac12}dt;L^\infty)}\lesssim 1+(\log A)^{-1}\log(t_2/t_1)
\end{equation}
for $0<t_1\leq t_2\leq1$, where one will recall the large parameter $A$ from \S\ref{parametersubsection}.
\end{proposition}

We remark that we almost never take advantage of the exponential factor in \eqref{vweightedlinftybound}, since all of the nontrivial behavior of $v^{(i)}$ is confined to times $t\in[0,\N{0}^{-2}]$.

\begin{proof}
The divergence-free property is immediate from Definition~\ref{principalpartdefinition} because the constituent parts of $v^{(i)}$, namely $v_k$ and $\overline v_k$, appear inside of $\curl$ and $\mathbb P$, respectively.

Next we proceed to the estimates, suppressing the dependence of $v^{(i)}$ on $i$ for conciseness. We recall the coefficient functions $a_{j,k}$ from Definition~\ref{dataiteratedefinition}. With \eqref{psidatabounds}, \eqref{cutoffbounds}, and \eqref{GammaboundII}, one can use the Leibnitz rule to estimate
\eq{\label{abounds}
\|\grad^ma_{j,k}\|_{L^\infty}&\lesssim_m \N{k-1}^m\N{k}
}
for $m\geq0$. It follows from \eqref{abounds} and $\eqref{mikadoprofilebound}$ that
\eq{\label{vkbounds}
\|\grad^mv_k\|_{L^\infty}&\lesssim \N{k}^{1+m}e^{-\N{k}^{2}t}.
}
For $\overline v_k$, one runs into the difficulty that $\mathbb P$ is not bounded on $L^\infty$. Instead\footnote{We expect this difficulty could be avoided with more careful analysis, but doing so is not necessary in the present work. Note, importantly, that the resulting loss of $\N{k}^{\beta_1}$ does not exist at $t=0$, as one can see from the identity $\overline v_k(0,x)=v_k(0,x)$.}, we carry out the estimate in the H\"older space $C^{\beta_1}$ where $\beta_1=\frac{b-1}4\wedge\frac12$. By \eqref{cutoffbounds}, \eqref{abounds}, and \eqref{psidatabounds},
\eq{\label{vkbarbounds}
\|\grad^m\overline v_k\|_{L^\infty}&\lesssim\|\grad^m\overline v_k\|_{C^{\beta_1}}\lesssim \N{k}^{1+{\beta_1}+m} e^{-\N{k+1}^{2}t}.
}
Note that the homogeneous $C^{\beta_1}$ norm controls $L^\infty$ in this case because $\overline v_k$ has zero average on $\Td$. To see \eqref{vweightedlinftybound}, we combine the bounds on $v_k$ and $\overline v_k$ to find
\eq{\label{quack}
\|\grad^mv(t)\|_{L_x^\infty}&\lesssim_m \sum_{k\geq0}(\N{k}^{1+m}e^{-\N{k}^2t}+\N{k}^{1+{\beta_1}+m} e^{-\N{k+1}^{2}t}).
}
We analyze the two terms separately. Because $\N{k}$ grows rapidly, the first term in the sum is controlled by the largest summand; when $t\leq \N{0}^{-2}$, the largest possible value occurs near $\N{k}=t^{-\frac12}$, leading to a sum of $O(t^{-(1+m)/2})$. When $t>\N{0}^{-2}$, the summands decrease monotonically leading to an upper bound of $\N{0}^{1+m}e^{-\N{0}^2t}$. Using that $t\N{0}^2>1$, this is controlled by $O_m(t^{-(1+m)/2}e^{-\N{0}^2t/O_m(1)})$ due to the exponential decay. This yields the desired bound \eqref{vweightedlinftybound} for the first term in \eqref{quack}. The second is estimated in the same way: the sum is controlled by its largest contribution which occurs when $\N{k+1}$ is near $t^{-\frac12}$. In total, one finds that \eqref{quack} leads to
\eqn{
\|\grad^mv(t)\|_{L^\infty}&\lesssim_mt^{-\frac{1+m}2}+\min\big\{t^{-\frac{1+{\beta_1}+m}{2b}},\N{0}^{1+\beta_1+m}e^{-\N{1}^2t}\big\}.
}
By definition of $\beta_1$, we have $\frac{1+{\beta_1}+m}{2b}\leq\frac{1+m}2$. It follows that when $t\leq \N{0}^{-2}$, $t^{-\frac{1+{\beta_1}+m}{2b}}$ can be absorbed into the first term. When $t>\N{0}^{-2}$, we have
\eqn{
\N{0}^{1+\beta_1+m}e^{-\N{1}^2t}&\lesssim_m \N{0}^{1+\beta_1+m}(\N{1}^2t)^{-\frac{1+m}2}e^{-\N{0}^2t}=\left(\frac{\N{0}}{\N{1}}\right)^{m}\frac{\N{0}^{1+\beta_1}}{\N{1}}t^{-\frac{1+m}2}e^{-\N{0}^2t}.
}
Recall that $\N{0}<\N{1}$ and, more strongly, $\N{0}^{1+\beta_1}<\N{1}$ because $\beta_1<b-1$. We conclude that the second contribution in \eqref{quack} can be absorbed into $t^{-\frac{1+m}2}$ as well. This completes the proof of \eqref{vweightedlinftybound}.

Next we turn to \eqref{vl2intimelinftyinspacebound}. We remark that it suffices to prove the $L^1$ estimate because the $L^2$ estimate is immediate by H\"older's inequality, combining the $L^1$ bound with \eqref{vweightedlinftybound}. We decompose
\eqn{
\|v\|_{L^1_t([t_1,t_2],t^{-\frac{1}{2}}dt;L^\infty_x)} &\leq \sum_{k}\int_{t_1}^{t_2}\left(\|v_k(t)\|_{L^\infty}\mathbf{1}_{t\leq \N{k}^{-2}}+\|v_k(t)\|_{L^\infty}\mathbf{1}_{t>\N{k}^{-2}}+\|\overline v_k(t)\|_{L^\infty}\right)t^{-\frac{1}{2}}dt\\
&\lesssim \sum_{k: \N{k} \leq t_1^{-1/2}}\int_{t_1}^{\N{k}^{-2}\wedge t_2} \N{k} t^{-\frac12}dt + \sum_{k: \N{k}\geq t_2^{-1/2}} \int_{\N{k}^{-2}\vee t_1}^{t_2} \N{k} e^{-\N{k}^2t} t^{-\frac12}dt\\
&\quad+\sum_k\int_{t_1}^{t_2}\N{k}^{1+\beta_1}e^{-\N{k+1}^2t}t^{-\frac12}dt.
}
Let $I$, $II$, and $III$ be the three terms in the final lines. For $I$ we simply use the fundamental theorem of calculus to find
$$I \leq \sum_{k: \N{k} \leq t_1^{-1/2}} \N{k} (\N{k}^{-2}\wedge t_2)^{1/2}\leq \sum_{k: \N{k} \leq t_2^{-1/2}} \N{k} t_2^{1/2} + \sum_{k: t_2^{-1/2} \leq \N{k} \leq t_1^{-1/2}} 1.$$
Clearly the first contribution is $O(1)$, while the second is the number of $\N{k}$ in the interval $(t_2^{-1/2}, t_1^{-1/2})$. We estimate this in a wasteful way: $\N{k}$ is more lacunary than the geometric sequence with ratio $A^{b-1}$ because $\N{k+1}/\N{k}=A^{\gamma(b^{k+1}-b^k)}\geq A^{\gamma(b-1)}$. Therefore,
\eqn{
\#\{k:\N{k}\in[t_2^{-\frac12},t_1^{-\frac12}]\}\leq\log_{A^{\gamma(b-1)}}\left(\frac{t_1^{-1/2}}{t_2^{-1/2}}\right)=(2\gamma(b-1)\log A)^{-1}\log(t_2/t_1)
}
which yields the desired estimate on $I$. For $II$, we have
\eqn{ II &\leq \sum_{k: \N{k} \geq t_2^{-1/2}} \N{k}^{-1}(\N{k}^{-2}\vee t_1)^{-\frac12}e^{-\N{k}^2(\N{k}^{-2}\vee t_1)}\\
&\lesssim \sum_{k: t_2^{-\frac12}\leq \N{k}\leq t_1^{-\frac12}}1+\sum_{k:\N{k}>t_1^{-\frac12}}\N{k}^{-1}t_1^{-\frac12}e^{-\N{k}^2t_1}
}
which is once again $O(1+\#\{k:\M{k}\in[t_2^{-\frac12},t_1^{-\frac12}]\})$, and we conclude as in $I$. For $III$, we bound $t^{-\frac12}$ in the trivial way and evaluate the integral to find
\eqn{
III&\lesssim\sum_k\N{k}^{1+\beta_1}\N{k+1}^{-2}t_1^{-\frac12}e^{-\N{k+1}^2t_1}\lesssim t_1^{\frac12-\frac{1+\beta_1}{2b}}.
}
Recall that $\beta_1<b-1$ by definition so the power is positive. As we assumed $t_1\leq1$, we conclude that $III=O(1)$.
\end{proof}

\section{Construction of the Perturbation}\label{perturbationsection}

To summarize our progress to this point, we have constructed two vector fields $v^{(1)}$ and $v^{(2)}$ with common initial data $U^0$ and reasonable estimates for $t>0$. The purpose of this section is to show that they nearly solve the Navier--Stokes equations and can be perturbed into exact solutions.

\subsection{Estimates on the Residual}\label{residualsubsection}

For $i\in\{1,2\}$, we shall define the residuals $F^{(i)}\in C^\infty((0,1]\times\Td;S^{3\times3})$ in divergence form which satisfy
\eq{\label{Frequirement}
-\mathbb P\div F^{(i)}=\dd_tv^{(i)}-\Delta v^{(i)}+\mathbb P\div v^{(i)}\otimes v^{(i)}.
}
Note that $\mathbb P\div$ does not have a unique right inverse because $F^{(i)}$ has too many degrees of freedom. In fact, an arbitrary choice of $F^{(i)}$ satisfying \eqref{Frequirement} would probably lack the required smallness to complete the construction. In the following proposition, we fix a suitable pair $F^{(1)}$, $F^{(2)}$.

In order to state the proposition, we introduce two important parameters: the subcriticality parameter $\alpha\in(0,\frac18)$ and the H\"older exponent $\kappa\in(0,\frac12 - \frac1{{2\gamma}}-2\alpha)$. We also recall the growth rate $A$ for the frequency scales $\N{k},\M{k}$ from \S\ref{parametersubsection}. The objective is to find an $F^{(i)}$ that can be estimated in, say, $L^\infty$ in such a way that the upper bound has slightly subcritical dependence on time. There are two types of terms that appear: those with subcritical \emph{amplitude} with an upper bound of the form $\N{k}^{1-c}\exp(-\N{k}^2t)$, and those with subcritical \emph{decay} with an upper bound similar to $\N{k}\exp(-\N{k}^{2+c}t)$ for some $c>0$. In either case, we can extract some smallness in $C^{1,\kappa}$ for $\kappa>0$ as above with a subcritical weight in time.

\begin{proposition}\label{Festimateproposition}
For any $\epsilon_0>0$ and all sufficiently large $A>1$ (depending on $\epsilon_0$), there exist $F^{(1)},\,F^{(2)}\in C^\infty((0,1]\times\Td;S^{3\times3})$ obeying \eqref{Frequirement} and
\eqn{
\sup_{t\in(0,1]}\big(t^{1-\alpha}\|F^{(i)}\|_{L^\infty}+t^{\frac32-\alpha}\|\grad F^{(i)}\|_{C^\kappa}\big)&\leq\epsilon_0,\quad i\in\{1,2\}.
}
\end{proposition}

\begin{proof}
Before beginning to define and estimate the residual, we extract the part of $v_k$ that will not be perturbative, decomposing $v_k=v_k^p+v_k^e$ where
\eqn{
v_k^p&=\N{k}^2\sum_j|\eta_j|^2a_{j,k}\Psi_{j,k},\\
v_k^e&=v_k^{e,1}+v_k^{e,2}+v_k^{e,3}+v_k^{e,4},
}
having defined
\eqn{
v_k^{e,1}&=(\phi_k-\delta)*\curl\curl\sum_ja_{j,k}\Psi_{j,k},\\
v_k^{e,2}&=-\N{k}^{-2}\sum_j\Delta(a_{j,k}\tilde\varphi_j(\M{k}x))\sin(\N{k}(x-x_j)\cdot\eta_j)\theta_j\exp(-|\eta_j|^2\N{k}^2t),\\
v_k^{e,3}&=-2\N{k}^{-1}\sum_j\eta_j\cdot\grad(a_{j,k}\tilde\varphi_j(\M{k}x))\cos(\N{k}(x-x_j)\cdot\eta_j)\theta_j\exp(-|\eta_j|^2\N{k}^2t)\\
v_k^{e,4}&=\N{k}^{-2}\grad\sum_j\theta_j\cdot\grad a_{j,k}\tilde\varphi_j(\M{k}x)\sin(\N{k}(x-x_j)\cdot\eta_j)\exp(-|\eta_j|^2\N{k}^2t).
}
Here, $v_k^{e,1}$ captures the effect of the mollifier, while the rest consists of the terms arising when $\curl\curl=-\Delta+\grad\div$ acts on $v_k$. Note that for $v_k^{e,4}$, we have used \eqref{steadyeuler}. The dominant term is $v_k^p$ which appears when the Laplacian falls on the sine in $\Psi_{j,k}$. Sharp estimates for $v_k^p$ are immediate from \eqref{mikadoprofilebound} and \eqref{abounds},
\eq{\label{vkpbounds}
\|\grad^mv_k^p\|_{L^\infty}&\lesssim_m \N{k}^{1+m}\exp(-\N{k}^2t).
}
The ``error part'' $v_k^{e}$ is slightly more involved to bound because we require smallness in a subcritical norm. By \eqref{abounds},
\eqn{
\|\grad^mv_k^{e,1}\|_{L^\infty}&\lesssim_m \N{k}^{\frac43+m}\N{k+1}^{-\frac13}\exp(-\N{k}^2t)\\
\|\grad^mv_k^{e,2}\|_{L^\infty}&\lesssim_m \N{k}^{-1+m}\M{k}^2\exp(-\N{k}^2t)\\
\|\grad^mv_k^{e,3}\|_{L^\infty}&\lesssim_m \N{k}^{m}\M{k}\exp(-\N{k}^2t)\\
\|\grad^mv_k^{e,4}\|_{L^\infty}&\lesssim_m \N{k-1}\N{k}^{m}\exp(-\N{k}^2t)
}
for all $m\geq0$. Each of the right-hand sides is of the form $A^{-c(b,\gamma)b^k}\M{k}\N{k}^m\exp(-\N{k}^2t)$ for some $c(b,\gamma)\geq0$, and in fact $c=0$ is achieved for $v_k^{e,3}$ (see \S\ref{parametersubsection} for the relevant parameter relations), leading to 
\eq{\label{vkebounds}
\|\grad^mv_k^e\|_{L^\infty}&\lesssim_m \M{k}\N{k}^m\exp(-\N{k}^2t).
}

To define $F^{(1)}$ and $F^{(2)}$, we need to further decompose the interaction
\eqn{
v_{k}^p\otimes v_{k}^p&=\N{k}^4\sum_{j_1,j_2}|\eta_{j_1}|^2|\eta_{j_2}|^2a_{j_1,k}a_{j_2,k}\Psi_{j_1,k}\otimes\Psi_{j_2,k}=\mathcal N_{k,1}+\mathcal N_{k,2}+\mathcal N_{k,3}
}
where
\eqn{
\mathcal N_{k,1}&=\sum_ja_{j,k}|\eta_j|^4e^{-2|\eta_j|^2\N{k}^2t}a_{j,k}^2\theta_j\otimes\theta_j,\\
\mathcal N_{k,2}&=\sum_j|\eta_j|^4e^{-2|\eta_j|^2\N{k}^2t}a_{j,k}^2\big(\tilde\varphi_j^2(\M{k}x)\sin^2(\N{k}(x-x_j)\cdot\eta_j)-A_{j,k}\big)\theta_j\otimes\theta_j,\\
\mathcal N_{k,3}&=\N{k}^4\sum_{j_1\neq j_2}|\eta_{j_1}|^2|\eta_{j_2}|^2a_{j_1,k}a_{j_2,k}\Psi_{j_1,k}\otimes\Psi_{j_2,k}
}
where $A_{j,k}$ was defined in \eqref{l2normalizedpipes}. $\mathcal N_{k,1}$ is the leading part that acts as a ``force'' on $\overline v_{k-1}$, while $\mathcal N_{k,2}$ will be perturbative. Meanwhile, $\mathcal N_{k,3}$ consists of cross terms which vanish identically due to the pipe separation assumption \eqref{pipesseparated}.

Let us also define the tensor fields
\eqn{
R_k&=\phi_k*\mathcal D\sum_ja_{j,k}\Psi_{j,k}\\
\overline R_k&=2^{-1}\N{k+1}^{-2}\sum_j A_{j,k+1}|\eta_j|^2e^{-2|\eta_j|^2 \N{k+1}^2 t}a_{j,k+1}^2(x)\theta_j\otimes\theta_j
}
which are constructed so that
\eq{\label{Rproperty}
\mathbb{P}\div R_k=v_k,\quad\mathbb{P}\div \overline R_k=\overline v_k.
}
At this point we specialize to the case $i=1$ in which, we recall from Definition~\ref{principalpartdefinition}, $v^{(1)}=v_0+\overline v_1+v_2+\overline v_3+\cdots$. To simplify this expansion, we define the consolidated parts of the velocity
\eqn{
v^p\coloneqq \sum_{k\geq0\text{ even}}v_k^p,\quad\quad v^e\coloneqq\sum_{k\geq0\text{ even}}v_k^e,\quad\quad \overline v\coloneqq\sum_{k\geq1\text{ odd}}\overline v_k
}
so that we may write $v^{(1)}=v^p+v^e+\overline v$. From \eqref{vkpbounds} we easily have
\eq{\label{vpbounds}
\|\grad^mv^p\|_{L^\infty}&\lesssim_m\sum_k\N{k}^{1+m}\exp(-\N{k}^2t)\lesssim_m t^{-\frac{1+m}{2}}.
}
Next, from \eqref{vkebounds},
\eqn{
\|\grad^mv^e\|_{L^\infty}&\lesssim_m \sum_{k\geq0}\M{k}\N{k}^m\exp(-\N{k}^2t)\\
&\lesssim (\sup_k\M{k}/\N{k})^\frac12\sum_{k\geq0}(\M{k}\N{k})^\frac12\N{k}^m\exp(-\N{k}^2t).
}
The factor in front is comparable to $\sup_kA^{-(\gamma-1)b^k/2}=A^{-(\gamma-1)/2}$. With a sufficiently large choice of $A$ (depending on $\gamma$), this is majorized by $\epsilon_0/C_3$ where $C_3$ can be chosen as needed to absorb incidental constants that appear. With $t>0$ fixed, we estimate the sum. Observe that the summands increase rapidly until $\N{k}\sim t^{-\frac12}$, then they decrease rapidly. As a result, the sum is controlled by the largest term. We can crudely provide an upper bound by letting $\N{k}$ run over $(0,\infty)$ and maximizing; clearly the maximum is attained at $\N{k}\sim t^{-\frac12}$. We arrive at
\eq{\label{vebound}
t^\frac12\|v^e\|_{L^\infty}+t\|\grad v^e\|_{C^\kappa}&\lesssim C_0^{-1}\epsilon_0t^{\frac14-\frac1{4\gamma}-\frac\kappa2}\leq C_0^{-1}\epsilon_0t^\alpha
}
by \eqref{parameterinequality}, since $\M{k}=\N{k}^{1/\gamma}$ and $t\leq1$. Similarly, using \eqref{vkbarbounds},
\eqn{
\|\grad^m\overline v\|_{L^\infty}&\lesssim_m \sum_{k\geq1}\N{k}^{1+\beta_1+m}e^{-\N{k+1}^2t/O(1)}\lesssim(\sup_k\N{k}/\N{k+1})^\frac14\sum_{k\geq1}\N{k}^{\frac34+\beta_1+m}\N{k+1}^\frac14e^{-\N{k+1}^2t/O(1)}
}
and therefore, by analogous reasoning and a sufficiently large choice of $A$ depending on $b$,
\eq{\label{vbarbound}
t^\frac12\|\overline v\|_{L^\infty}+t\|\grad\overline v\|_{C^\kappa}&\lesssim C_3^{-1}\epsilon_0(t^{\frac38-\frac{3/4+\beta_1}{2b}}+ t^{\frac{1}{2}-\frac{1}{2b}-\frac{\kappa}{2}}t^{\frac38-\frac{3/4+\beta_1}{2b}})\leq C_3^{-1}\epsilon_0t^\alpha,
}
where $\beta_1=\min\{\frac{b-1}4,\frac12\}$ as before, and estimating using \eqref{parameterinequality} that $\frac38-\frac{3/4+\beta_1}{2b}\geq \frac14(1-b^{-1})>\alpha$ and that $\kappa<1-b^{-1}$. A nearly identical argument will be used for every contribution to extract some smallness in exchange for a bit of criticality so we shall not repeat the details each time.

Finally we can construct $F^{(1)}$:
\eqn{
F^{(1)}\coloneqq F_1^{(1)}+F_2^{(1)}+F_3^{(1)}+F_4^{(1)}
}
where
\eqn{
F_1^{(1)}&=-\sum_{k\geq0\text{ even}}(\dd_t-\Delta)R_k\\
F_2^{(1)}&=\sum_{k\geq1\text{ odd}}\Delta\overline R_k\\
F_3^{(1)}&=-\mathcal R\div\mathcal N_{0,1}-\sum_{k\geq1
\text{ odd}}(\dd_t\overline R_k+\mathcal N_{k+1,1})\\
F_4^{(1)}&=-v^{(1)}\otimes v^{(1)}+\sum_{k\geq0\text{ even}}(v_k^p\otimes v_k^p-\mathcal R\div\mathcal N_{k,2})
}
where $\mathcal R$ is the anti-divergence operator defined in Lemma~\ref{Rdefinition}. Let us briefly verify that $F^{(1)}$ defined in this way satisfies \eqref{Frequirement}. We compute
\eq{
-\mathbb P\div F^{(1)}&=\sum_{k\geq0\text{ even}}(\dd_t-\Delta)\mathbb P\div R_k-\sum_{k\geq1\text{ odd}}\Delta\mathbb P\div \overline R_k+\mathbb P\div\mathcal N_{0,1}\nonumber\\
&\quad+\sum_{k\geq1\text{ odd}}(\dd_t\mathbb P\div \overline R_k+\mathbb P\div\mathcal N_{k+1,1})+\mathbb P\div v^{(1)}\otimes v^{(1)}\nonumber\\
&\quad-\sum_{k\geq0\text{ even}}\mathbb P\div v_k^p\otimes v_k^p+\sum_{k\geq0\text{ even}}\mathbb P\div\mathcal N_{k,2}\nonumber\\
&=\sum_{k\geq0\text{ even}}(\dd_t-\Delta)v_k+\sum_{k\geq1\text{ odd}}(\dd_t-\Delta)\overline v_k+\mathbb P\div v^{(1)}\otimes v^{(1)}\label{firstline}\\
&\quad+\mathbb P\div\Big(\mathcal N_{0,1}+\sum_{k\geq1\text{ odd}}\mathcal N_{k+1,1}-\sum_{k\geq0\text{ even}}( v_k^p\otimes v_k^p-\mathcal N_{k,2})\Big)\label{secondline}
}
using \eqref{Rproperty} and the fact that $\div\mathcal R$ acts as the identity on zero-average fields (which, indeed, is the case throughout by the divergence theorem). Clearly the expression on line \eqref{firstline} is precisely $\dd_tv^{(1)}-\Delta v^{(1)}+\mathbb P\div v^{(1)}\otimes v^{(1)}$. Meanwhile, the expression on the next line \eqref{secondline} identically vanishes because $v_k^p\otimes v_k^p=\mathcal N_{k,1}+\mathcal N_{k,2}$. Thus we conclude that \eqref{Frequirement} indeed holds for $F^{(1)}$.

Now we proceed to prove the claimed estimate on $F^{(1)}$. Consider the contribution to $F_1^{(1)}$ from a fixed $j\in\{1,2,\ldots,6\}$ and $k$ even. The key observation is that the heat operator exactly annihilates the $\sin(\N{k}(x-x_j)\cdot\eta_j)\exp(-|\eta_j|^2\N{k}^2t)$ factor within $\Psi_{j,k}$; thus one is left with the remaining terms from the Laplacian,
\eqn{
\N{k}^{-2}|\phi_k*\mathcal D\big(\Delta(\tilde\varphi_ja_{j,k})\sin\Theta+2\eta_j\cdot\grad(\tilde\varphi_ja_{j,k})\N{k}\cos\Theta\big)|e^{-|\eta_j|^2\N{k}^2t}
\lesssim \M{k}\N{k}e^{-\N{k}^2t}
}
by \eqref{abounds}, having written $\Theta$ in place of $\N{k}(x-x_j)\cdot\eta_j$ for brevity.
Once again by \eqref{abounds}, additional derivatives cost a factor of $\N{k}$ each and we obtain
\eqn{
\|\grad^mF_1^{(1)}\|_{L^\infty}&\lesssim \sum_{k\geq0\text{ even}}\M{k}\N{k}^{1+m}\exp(-\N{k}^2t)
}
for $m\in\{0,1\}$. This implies
\eq{\label{F1estimate}
t\|F_1^{(1)}\|_{L^\infty}+t^{\frac32}\|\grad F_1^{(1)}\|_{C^\kappa}\lesssim C_3^{-1}\epsilon_0t^{\frac14-\frac1{4\gamma}}\leq C_3^{-1}\epsilon_0t^\alpha
}
by \eqref{parameterinequality} and identical reasoning as above to extract the $\epsilon_0$ factor.

For $F_2^{(1)}$, we use the crude upper bound coming from \eqref{abounds},
\eqn{
\frac12\N{k+1}^{-2}|\Delta \sum_j A_{j,k+1}|\eta_j|^2e^{-2|\eta_j|^2 \N{k+1}^2 t}a_{j,k+1}^2(x)\theta_j\otimes\theta_j|&\lesssim \N{k}^2e^{-\N{k+1}^2t}.
}
Once again, additional derivatives cost $\N{k}$ so
\eqn{
\|\grad^mF_2\|_{L^\infty}&\lesssim \sum_{k\geq1\text{ odd}}\N{k}^{2+m}\exp(-\N{k+1}^2t).
}
By the same reasoning as before, we may extract some smallness from the fast time decay and obtain
\eqn{
t\|F_2^{(1)}\|_{L^\infty}+t^{\frac32}\|\grad F_2^{(1)}\|_{C^\kappa}\lesssim C_3^{-1}\epsilon_0t^{\frac12-\frac1{2b}}\leq C_3^{-1}\epsilon_0t^\alpha
}
by \eqref{parameterinequality} and a sufficiently large choice of $A$.

Next, we claim that $F_3^{(1)}$ vanishes identically. Indeed, $\mathcal N_{0,1}$ is a constant because $a_{j,0}$ are constants, so it is annihilated by $\div$. Then, it is a trivial calculation from the definitions that $\dd_t\overline R_k+\mathcal N_{k+1,1}$ identically vanishes for all $k\geq1$.

Moving to $F_4^{(1)}$, recall the decomposition $v^{(1)}=v^p+v^e+\overline v$ from which we can expand
\eqn{
F_4^{(1)}&=-\Big(\sum_{k\neq m\geq0\text{ even}}v_k^p\otimes v_m^p\Big)-v^{(1)}\otimes(\overline v+v^e)-(\overline v+v^e)\otimes v^p-\sum_{k\geq0\text{ even}}\mathcal R\div\mathcal N_{k,2}\\
&\eqqcolon F_{4,1}+F_{4,2}+F_{4,3}+F_{4,4}.
}
We immediately have from \eqref{vweightedlinftybound}, \eqref{vpbounds}, \eqref{vebound}, \eqref{vbarbound}, and the product rule for $C^{1,\kappa}\cap L^\infty$ that $F_{4,2}$ and $F_{4,3}$ obey the desired estimates.

Next consider $F_{4,1}$. By \eqref{vkpbounds},
\eqn{
\|F_{4,1}(t)\|_{L^\infty}&\leq\sum_{k\neq m}\|v_k^p\otimes v_m^p\|_{L^\infty}\lesssim \sum_{k\neq m}\N{k}\N{m}\exp(-(\N{k}^2+\N{m}^2)t)\\
&\lesssim\sum_{m}\left(\N{m}\exp(-\N{m}^2t)\sum_{k<m}\N{k}\right)\\
&\lesssim \sum_m\N{m-1}\N{m}\exp(-\N{m}^2t)\\
&\lesssim C_3^{-1}\epsilon_0t^{-\frac34-\frac1{4b}}
}
where, to obtain the second line from the first, we invoke the symmetry of $k$ and $m$ and sum over only the indices $k<m$ without losing of generality.

Similarly, for the derivative,
\eqn{
\|\grad F_{4,1}(t)\|_{C^\kappa}&\lesssim \sum_{k\neq m}(\N{k}+\N{m})^{1+\kappa}\N{k}\N{m}\exp(-\N{k\vee m}^2t)\\
&\lesssim\sum_{m}\left(\N{m}^{2+\kappa}\exp(-\N{m}^2t)\sum_{k<m}\N{k}\right)\\
&\lesssim \sum_m\N{m-1}\N{m}^{2+\kappa}\exp(-\N{m}^2t/O(1))\\
&\lesssim C_3^{-1}\epsilon_0 t^{-\frac{5}{4}-\frac{1}{4b}-\frac{\kappa}{2}}.
}
With \eqref{parameterinequality} and the fact that $\kappa<\frac{1}{2}-\frac{1}{2b}-2\alpha$, we conclude that
\eqn{
t\|F_{4,1}(t)\|_{L^\infty}+t^\frac32\|F_{4,1}(t)\|_{L^\infty}\lesssim C_3^{-1}\epsilon_0t^{\alpha}.
}
Finally we estimate $F_{4,4}$:
\eqn{
-\mathcal R\div\sum_{k\geq0\text{ even}}\sum_j|\eta_j|^4e^{-2|\eta_j|^2\N{k}^2t}a_{j,k}^2\big(\tilde\varphi_j^2(\M{k}x)\sin^2(\N{k}(x-x_j)\cdot\eta_j)-A_{j,k}\big)\theta_j\otimes\theta_j.
}
Let us fix a particular $j,k$. The first observation is that when $\div$ falls on the expression in parentheses, the result vanishes due to \eqref{steadyeuler}. What remains is
\eqn{
F_{4,4}=-\sum_{k\geq0\text{ even}}\sum_j|\eta_j|^4e^{-2|\eta_j|^2\N{k}^2t}F_{4,4,j,k}:\theta_j\otimes\theta_j
}
where $F_{4,4,j,k}$ can be written in coordinates as
\eqn{
F_{4,4,j,k}^{j_1j_2j_3j_4}\coloneqq\mathcal R^{j_1j_2j_3}\Big(\dd_{j_4}(a_{j,k}^2)\big(\tilde\varphi_j^2(\M{k}x)\sin^2(\N{k}(x-x_j)\cdot\eta_j)-A_{j,k}\big)\Big).
}
We will omit the coordinates henceforth because they play no role in the estimates. We observe that $F_{4,4,j,0}$ vanishes identically because $a_{j,0}$ is a constant.

Recall by definition that $\N{k}$ is an integer multiple of $\M{k}$; thus we may write
\eqn{
h_{j,k}(\M{k}x)=\tilde\varphi_j^2(\M{k}x)\sin^2(\N{k}(x-x_j)\cdot\eta_j)-A_{j,k}
}
for a $2\pi$-periodic function $h_{j,k}\in C^\infty(\Td;\mathbb R)$. Most importantly, due to \eqref{l2normalizedpipes}, $h_{j,k}$ has zero average on $\Td$. Thus, taking a Fourier series of $h_{j,k}$, we obtain
\eq{\label{F44ikseries}
F_{4,4,j,k}(x)&=\sum_{\xi\in\Zd\setminus0}\hat h_{j,k}(\xi)\mathcal R\Big(\grad(a_{j,k}^2)(x)e^{i\M{k}\xi\cdot x}\Big).
}
We need to estimate the Fourier coefficients of $h_{j,k}=\tilde\varphi_j^2(x)\sin^2(\N{k}\M{k}^{-1}(x-\M{k}x_j)\cdot\eta_j)-A_{j,k}$. Obviously the $\sin^2$ factor has Fourier support at just $\xi=0$ and $\xi=\pm 2\N{k}\M{k}^{-1}\eta_j$. Thus
\eqn{
\hat h_{j,k}(\xi)&=\Big(s_0\mathcal F(\tilde{\varphi}_j^2)(\xi)+s_1\mathcal F(\tilde{\varphi}_j^2)(\xi+2\N{k}\M{k}^{-1}\eta_j)+s_1\mathcal F(\tilde{\varphi}_j^2)(\xi-2\N{k}\M{k}^{-1}\eta_j)\Big)\mathbf1_{\xi\neq0}
}
where $s_0$ and $s_1$ are the Fourier coefficients of $\sin^2(\N{k}\M{k}^{-1}(x-M_kx_j)\cdot\eta_j)$ at $\xi=0$ and $\xi=\pm \N{k}\M{k}^{-1}\eta_j$, respectively. One readily computes $|s_0|=4\pi^3$ and $|s_1|=2\pi^3$. This implies
\eqn{
|\hat h_{j,k}(\xi)|&\lesssim\langle\xi\rangle^{-10}+\langle\xi+2\N{k}\M{k}^{-1}\eta_j\rangle^{-10}+\langle\xi-2\N{k}\M{k}^{-1}\eta_j\rangle^{-10}.
}
Finally, we can estimate \eqref{F44ikseries} with an application of Lemma~\ref{stationaryphase} with H\"older parameter $\beta$ freely chosen in the range
\eqn{
0<\beta<1-\gamma(4\alpha+b^{-1})
}
where the coefficient function ``$a$'' in the lemma becomes $\grad(a_{j,k}^2)$. Note that this range is non-empty by definition of $\gamma$. Using \eqref{abounds}, we have, for all $\xi\in\Zd\setminus0$,
\eqn{
\|\mathcal R\big(\grad(a_{j,k}^2)(x)e^{i\M{k}\xi\cdot x}\big)\|_{C^\beta}&\lesssim_m \N{k-1}\N{k}^2(\M{k}^{-1+\beta}+\M{k}^{-m+\beta}\N{k-1}^m+\M{k}^{-m}\N{k-1}^{m+\beta}).
}
The third term is clearly dominated by the second and can be dropped. By choosing $m>b/(1-\gamma/b)$, we can make $(\N{k-1}/\M{k})^m\ll \M{k}^{-1}$, allowing the second term to be dropped as well. This leaves us with 
\eqn{
\|F_{4,4,j,k}\|_{C^\beta}&\lesssim \N{k-1}\M{k}^{-1+\beta}\N{k}^2\sum_{\xi\in\Zd}\Big(\langle\xi\rangle^{-10}+\langle\xi+2\N{k}\M{k}^{-1}\eta_j\rangle^{-10}+\langle\xi-2\N{k}\M{k}^{-1}\eta_j\rangle^{-10}\Big).
}
Clearly this is summable with a uniform constant; thus
\eqn{
\|F_{4,4}(t)\|_{L^\infty}&\lesssim\sum_{k\geq2\text{ even}}\N{k-1}\M{k}^{-1+\beta}\N{k}^2e^{-\N{k}^2t}\\
&\lesssim \sup_{k}(\N{k-1}\M{k}^{-1+\beta})^\frac12\sum_k\N{k-1}^\frac12\M{k}^{-(1-\beta)/2}\N{k}^{2}e^{-\N{k}^2t}\\
&\lesssim(\sup_k A^{(\frac\gamma{ b}-1+\beta)b^k})^\frac12t^{-\frac1{4b}+\frac{1-\beta}{4\gamma}-1}.
}
By definition of $\beta$, the power $\frac\gamma b-1+\beta<-4\alpha\gamma$ is negative so the prefactor can be made as small as needed with the choice of $A$, as usual. 

Finally, we estimate the gradient. Recall that $\grad\mathcal R$ is a Calder\'on--Zygmund operator so it is bounded on $C^\kappa(\Td)$. Thus, by \eqref{abounds} and interpolation,
\eqn{
\|\grad F_{4,4,j,k}\|_{C^\kappa}&\lesssim \|\grad(a_{j,k}^2)\big(\tilde\varphi_j^2(\M{k}x)\sin^2(\N{k}(x-x_j)\cdot\eta_j)-A_{j,k}\big)\|_{C^\kappa}\lesssim \N{k-1}\N{k}^{2+\kappa} .
}
Let $c_1\in(0,1-\frac{b(\kappa+2\alpha)}{b-1})$. Then
\eqn{
\|\grad F_{4,4}(t)\|_{C^\kappa}&\lesssim \sup_k(\N{k-1}/\N{k})^{c_1}\sum_{k\geq2\text{ even}}e^{-2\N{k}^2t/O(1)}\N{k-1}^{1-c_1}\N{k}^{2+\kappa+c_1}\lesssim t^{-\frac{1-c_1}{2b}-\frac{2+\kappa+c_1}2}.
}
We conclude
\eqn{
t\|F_{4,4}(t)\|_{L^\infty}+t^{\frac32}\|\grad F_{4,4}(t)\|_{L^\infty}&\lesssim C_3^{-1}\epsilon_0(t^{\frac{1-\beta}{4\gamma}-\frac1{4b}}+t^{\frac{1-\kappa-c_1}2-\frac{1-c_1}{2b}}).
}
Both powers are bounded from below by $\alpha$: the first because of the definitions of $\beta, b,$ and $\gamma$, and the second because of the definitions of $c_1, \kappa,$ and $b$.

With $F^{(1)}$ fully estimated, let us now address $F^{(2)}$ which corresponds to $v^{(2)}=\overline v_0+v_1+\overline v_2+v_3+\cdots$. We have an exactly analogous decomposition $v^{(2)}=v^p+v^e+\overline v$ where $v^p$, $v^e$, and $\overline v$ are as before, the only change being that ``even'' and ``odd'' are reversed in the definitions. The estimates \eqref{vpbounds}--\eqref{vbarbound} still hold by the same arguments. Now we can define $F^{(2)}$:
\eqn{
F^{(2)}\coloneqq F_1^{(2)}+F_2^{(2)}+F_3^{(2)}+F_4^{(2)}
}
where
\eqn{
F_1^{(2)}&=-\sum_{k\geq1\text{ odd}}(\dd_t-\Delta)R_k\\
F_2^{(2)}&=\sum_{k\geq0\text{ even}}\Delta\overline R_k\\
F_3^{(2)}&=-\sum_{k\geq0
\text{ even}}(\dd_t\overline R_k+\mathcal N_{k+1,1})\\
F_4^{(2)}&=-v^{(2)}\otimes v^{(2)}+\sum_{k\geq1\text{ odd}}(v_k^p\otimes v_k^p-\mathcal R\div\mathcal N_{k,2}).
}
From here, identical calculations to the ones for $F^{(1)}$ can be used to verify \eqref{Frequirement} and the claimed estimates for $F^{(2)}$.
\end{proof}

\subsection{Semigroup of the Linearization around the Principal Part}\label{semigroupsubsection}

Recall the subcriticality parameter $\alpha\in(0,\frac18)$ and H\"older exponent $\kappa\in(0,\frac12 - \frac1{{2\gamma}}-2\alpha)$. We consider the Banach space
\eqn{
Y=\{a\in C^0((0,1];C^{1,\kappa}(\Td;S^{3\times3})):\|a\|_{Y}\coloneqq \sup_{t\in(0,1]}(t^{1-\alpha}\|a\|_{L^\infty(\Td)}+t^{\frac32-\alpha}\|\grad a\|_{C^\kappa(\Td)})<\infty\}
}
in which we shall measure, for instance, the $F^{(i)}$ and $w^{(i)}\otimes w^{(i)}$. This choice is natural in light of Proposition~\ref{Festimateproposition}.

For $i\in\{1,2\}$, $a\in Y$, and $0<t'\leq t\leq1$, we define the semigroup $S^{(i)}(t,t')a$ as the solution of
\eqn{
\dd_tS^{(i)}(t,t')a-\Delta S^{(i)}(t,t')+2\mathbb P\div (v^{(i)}(t)\odot S^{(i)}(t,t')a)=0\\
S^{(i)}(t',t')=\mathbb P\div a(t').
}
There is no difficulty with $S^{(i)}(t,t')$ being well-defined and smooth because $v^{(i)}$ is smooth with uniform estimates on $[t',1]\times\Td$. For the reader's convenience, we give a proof of existence in Appendix~\ref{linearappendix}, or see~\cite[Theorem 21.3]{lemarie2002recent}. What is not so clear is whether one can obtain bounds that degenerate only mildly near $t'=0$, given that the smoothness of $v^{(i)}$ breaks down.

The effect of the following proposition is that $S^{(i)}(t,t')$ behaves comparably to $e^{(t-t')\Delta}\mathbb P\div$, with the slight loss $(t/t')^\epsilon$ ($\epsilon>0$ as small as desired) because $v^{(i)}$ does not belong to $L_t^2(\mathbb R_+;L_x^\infty)$ or $L^1(\mathbb R_+,t^{-\frac12}dt;L^\infty)$. If $v^{(i)}$ were an arbitrary vector field in, say, the Koch--Tataru space $X_{KT}$ and additionally obeyed \eqref{vweightedlinftybound}, we do not expect the loss would be so mild. It is essential that the loss is, in particular, smaller than the subcriticality window $(t')^\alpha$. Fortunately, due to $v^{(i)}$ being supported on a lacunary sequence of scales, quantities such as $\|v^{(i)}\|_{L_t^2([t',t];L_x^\infty)}$ can be arranged to deteriorate as slowly as desired as $t'/t\to0$; see the proof of Propostion~\ref{vestimatesproposition}.

\begin{proposition}\label{semigroupestimateproposition}
For all $i\in\{1,2\}$, $a\in Y$, and $0<t'\leq t\leq1$, we have
\eqn{
\|S^{(i)}(t,t')a\|_{L^\infty(\Td)}+(t-t')^\frac12\|\grad S^{(i)}(t,t')a\|_{C^\kappa(\Td)}\lesssim (t')^{-1+\alpha-\epsilon}t^{-\frac12+\epsilon}\|a\|_Y.
}
\end{proposition}

\begin{proof}
We suppress the dependence on $i$ and write $v$ for either $v^{(1)}$ or $v^{(2)}$. For a fixed $t'\in(0,1)$, the semigroup obeys the Duhamel formula
\eqn{
S(t,t')a&=e^{(t-t')\Delta}\mathbb P\div a(t')-2\int_{t'}^te^{(t-s)\Delta}\mathbb P\div(v(s)\odot S(s,t')a)ds\\
&\eqqcolon I+II+III
}
where $I$ is the linear propagator and $II$ and $III$ are the pieces of the integral over $[t',t'\vee(t/2)]$ and $[t'\vee(t/2),t]$ respectively.

Consider $I$. There are two ways to estimate: either by moving to $C^\kappa$ on which $e^{(t-t')\Delta}\mathbb P$ is uniformly bounded, or applying the $C^1\to L^\infty$ boundedness of the heat kernel from Lemma~\ref{heatlplemma}. We use whatever upper bound is more favorable and obtain
\eqn{
\|I\|_{L^\infty}&\lesssim \|\grad a(t')\|_{C^\kappa}\wedge((t-t')^{-\frac12}\|a(t')\|_{L^\infty})\\
&\lesssim (t')^{-\frac32+\alpha}\wedge((t-t')^{-\frac12}(t')^{-1+\alpha})\|a\|_{Y}\\
&=(t')^{-1+\alpha}(t'\vee (t-t'))^{-\frac12}\|a\|_Y\\
&\lesssim(t')^{-1+\alpha}t^{-\frac12}\|a\|_Y.
}
Next we have by \eqref{vweightedlinftybound}
\eqn{
\|II\|_{L^\infty}&\lesssim\int_{t'}^{t'\vee(t/2)}(t-s)^{-\frac12}\|v(s)\|_{L^\infty}\|S(s,t')a\|_{L^\infty}ds\\
& \lesssim t^{-\frac12}\int_{t'}^{t'\vee(t/2)} s^{-\frac12} \|v(s)\|_{L^\infty}\|S(s,t')a\|_{L^\infty} s^{\frac12}ds
}
where we have used the fact that the integral vanishes identically unless $t>2t'$, so $t/(t-s)\leq t/(t-t/2)=2$. Also by \eqref{vweightedlinftybound},
\eqn{
\|III\|_{L^\infty}&\lesssim \int_{t'\vee(t/2)}^t (t-s)^{-\frac12} \|v(s)\|_{L^\infty}\|S(s,t')a\|_{L^\infty} ds\\
&\lesssim t^{-\frac12} \int_{t'\vee(t/2)}^t (t-s)^{-\frac{1}{2}}\|v(s)\|_{L^\infty}\|S(s,t')a\|_{L^\infty} s^{\frac12}ds,
}
since $(t/s)^{1/2} \leq\sqrt2$ on the interval $s\in [t'\vee(t/2),t]$.
Letting $h(t)=t^\frac12\|S(t,t')a\|_{L^\infty}$ and summing the estimates for $I$, $II$, and $III$, we have arrived at
\eqn{
h(t)&\lesssim (t')^{-1+\alpha}\|a\|_Y+\int_{t'}^{t}(s^{-\frac12}+(t-s)^{-\frac12})\|v(s)\|_{L^\infty}h(s)ds.
}
This is in a form where we can apply Lemma~\ref{gronwalllemma} with $p=3$ (say), $g_1(s)=s^{-\frac12}\|v(s)\|_{L^\infty}$ and $g_2(s)=\|v(s)\|_{L^\infty}$ which lead to
\eqn{
h(t)&\lesssim (t')^{-1+\alpha}\|a\|_Y\exp(O(\|v\|_{L^1([t',t],t^{-\frac12}dt;L^\infty)}+\|s^\frac12 v\|_{L^\infty_{t,x}([t',t])}\|v\|_{L^2([t',t];L^\infty)}^2)).
}
By \eqref{vweightedlinftybound}--\eqref{vl2intimelinftyinspacebound}, we may estimate the norms in the exponent to obtain
\eqn{
h(t)&\lesssim(t')^{-1+\alpha}\|a\|_Y\exp(O(1+(\log A)^{-1}\log(t/t'))).
}
The exponential factor becomes $O((t/t')^{O(1/\log A)})$. Choosing $A$ sufficiently large depending on $\epsilon$, the power becomes smaller than $\epsilon/2$. Unpacking the definition of $h$ thus yields
\eq{\label{intermediateSlinftybound}
\|S(t,t')a\|_{L^\infty}&\lesssim t^{-\frac12}(t')^{-1+\alpha}(t/t')^{\frac\epsilon2}\|a\|_Y.
}

Next we address the $C^{1,\kappa}$ norm. To do so, we redefine the decomposition of $S(t,t')a$: $I$ stays the same but $II$ and $III$ split the integral into $[t',(t+t')/2]$ and $[(t+t')/2,t]$ respectively. Then, by a similar calculation using \eqref{besovheatestimate},
\eqn{
\|\grad I\|_{C^\kappa}&\lesssim ((t-t')^{-\frac12}\|\grad a(t')\|_{C^\kappa})\wedge((t-t')^{-1-\frac\kappa2}\|a(t')\|_{L^\infty})\\
&\lesssim ((t-t')^{-\frac12}(t')^{-\frac32+\alpha})\wedge((t-t')^{-1-\frac\kappa2}(t')^{-1+\alpha})\|a\|_Y\\
&\lesssim(t-t')^{-\frac12}(t')^{-1+\alpha}t^{-\frac{1+\kappa}2}\|a\|_Y,
}
having used the fact that $t'\vee(t-t')^{1+\kappa}\geq(t/2)^{1+\kappa}$ which can be seen by considering the cases $t'>t/2$ and $t'\leq t/2$ and using $t\leq1$. Next, by \eqref{besovheatestimate}, \eqref{vweightedlinftybound}, and \eqref{intermediateSlinftybound}, we have
\eqn{
\|\grad II\|_{C^\kappa}&\lesssim \int_{t'}^{(t+t')/2}(t-s)^{-1-\frac\kappa2}\|v(s)\|_{L^\infty}\|S(s,t')a\|_{L^\infty}ds\\
&\lesssim (t')^{-1+\alpha-\frac\epsilon2}\int_{t'}^{(t+t')/2}(t-s)^{-1-\frac\kappa2}s^{-1+\frac\epsilon2}ds\|a\|_Y\\
&\lesssim(t-t')^{-\frac\kappa2}(t')^{-1+\alpha-\frac\epsilon2}t^{-1+\frac\epsilon2}\|a\|_Y
}
and, additionally using \eqref{intermediateSlinftybound} and that $C^{1,\kappa}\cap L^\infty$ is a multiplication algebra,
\eqn{
\|\grad III\|_{C^\kappa}&\lesssim\int_{(t+t')/2}^t(t-s)^{-\frac12}(\|\grad v(s)\|_{C^\kappa}\|S(s,t')a\|_{L^\infty}+\|v(s)\|_{L^\infty}\|\grad S(s,t')a\|_{C^\kappa})ds\\
&\lesssim \int_{(t+t')/2}^t(t-s)^{-\frac12}(s^{-\frac32+\frac\epsilon2-\frac\kappa2}(t')^{-1+\alpha-\frac\epsilon2}\|a\|_Y+\|v(s)\|_{L^\infty}\|\grad S(s,t')a\|_{C^\kappa})ds.
}
The first term in the integrand contributes $O((t')^{-1+\alpha-\frac\epsilon2}t^{-1+\frac\epsilon2-\frac\kappa2}\|a\|_Y)$. Thus, defining $h(t)=(t-t')^{\frac12}t^\frac{1+\kappa}2\|\grad S(t,t')\|_{C^\kappa}$, we have
\eqn{
h(t)&\lesssim (t')^{-1+\alpha}(t/t')^\frac\epsilon2\|a\|_Y+\int_{(t+t')/2}^t(t-s)^{-\frac12}\|v(s)\|_{L^\infty}h(s)ds,
}
using that $t$ and $s$ are comparable on this time interval. We conclude once again by Lemma~\ref{gronwalllemma}, \eqref{vweightedlinftybound}, and \eqref{vl2intimelinftyinspacebound}.
\end{proof}

\subsection{Fixed Point Argument}\label{fixedpointsubsection}

We construct the perturbation $w^{(i)}$ by means of a fixed point argument on the Banach space\footnote{To keep the argument simple and to avoid spaces with piecewise-defined weights, we only execute the fixed point argument up to time $t=1$. Note that the question of existence of $w^{(i)}$ is only nontrivial when $t<\N{0}^{-2}$. Global-in-time existence follows easily from a small data argument; see \S\ref{proofoftheoremsection}.}
\eqn{
X=\{V\in C^0((0,1];C^{1,\kappa}(\Td;\Rd)):
\|V\|_X\coloneqq\sup_{t\in(0,1]}(t^{\frac12-\frac\alpha2}\|V\|_{L^\infty}+t^{1-\frac\alpha2}\|\grad V\|_{C^\kappa})<\infty\}.
}

Recall the parameter $\epsilon_0$ from Proposition~\ref{Festimateproposition} which quantifies the smallness of the residuals $F^{(i)}$. This smallness, combined with the zero initial data, enables us to construct the perturbations $w^{(i)}$ by a standard fixed point argument in $X$.

\begin{proposition}\label{wconstructionfixedpointproposition}
There exists $C_4>0$ such that for all $\epsilon_1\in(0,C_4^{-1})$ and sufficiently large $A>1$, there exist $w^{(1)},\,w^{(2)}\in B_X(0,\epsilon_1)$ such that $u^{(1)}=v^{(1)}+w^{(1)}$ and $u^{(2)}=v^{(2)}+w^{(2)}$ both satisfy \eqref{NSE}. Furthermore, $w^{(i)}(t)\to0$ in $\mathcal C^{-1+\frac\alpha2}$ as $t\to0$.
\end{proposition}

Note that at this point, we do not make any claim about the initial data of the $u^{(i)}$. The second statement will eventually play a role by, in effect, enforcing that the data is determined only by $v^{(i)}$. We point out that $w^{(i)}\in X$ alone does not imply any control near the initial time.

\begin{proof}
Once again in this proof, we suppress the dependence on $i$ of $v^{(i)}$, $w^{(i)}$, $F^{(i)}$, and $S^{(i)}$.

In order for $u=v+w$ to solve \eqref{NSE}, the perturbation $w$ should solve
\eqn{
\dd_tw-\Delta w+2\mathbb P\div v\odot w=\mathbb P\div (F-w\otimes w)\\
w|_{t=0}=0
}
where $F$ is as in Proposition~\ref{Festimateproposition}, in particular satisfying \eqref{Frequirement}. We solve for $w$ as a fixed point of the map
\eqn{
T(w)(t)=\int_0^tS(t,t')(F-w\otimes w)(t')dt'
}
where $S(t,t')$ is the semigroup of the linearization around $v$ defined in \S\ref{semigroupsubsection}. Recall as well the function space $Y$ defined in \S\ref{semigroupsubsection}. From the following estimates, it will follow that $T$ is a well-defined operator on $X$.

Based on the definitions of $X$ and $Y$, we have the elementary product estimate
\eqn{
\|w\otimes w\|_Y&\lesssim\sup_{0<t\leq1}(t^{\frac32-\alpha}\|w(t)\|_{L^\infty}\|\grad w(t)\|_{C^\kappa}+t^{1-\alpha}\|w(t)\|_{L^\infty}^2)\lesssim \|w\|_X^2.
}
Combining this with Propositions~\ref{Festimateproposition} and \ref{semigroupestimateproposition} and the fact that $\epsilon<\alpha$,
\eqn{
\|T(w)(t)\|_{L^\infty}&\lesssim t^{-\frac12+\epsilon}\int_0^t(t')^{-1+\alpha-\epsilon}dt'\|F-w\otimes w\|_Y\\
&\lesssim t^{-\frac12+\epsilon}\int_0^t(t')^{-1+\alpha-\epsilon}dt'(\|w\|_X^2+\|F\|_Y)\\
&\lesssim t^{-\frac12+\alpha}(\|w\|_X^2+\epsilon_0)
}
and
\eqn{
\|\grad T(w)(t)\|_{C^\kappa}&\lesssim t^{-\frac12+\epsilon}\int_0^t(t')^{-1+\alpha-\epsilon}(t-t')^{-\frac12}dt'\|F-w\otimes w\|_Y\\
&\lesssim t^{-1+\alpha}(\|w\|_{X}^2+\epsilon_0).
}
It can be arranged via Proposition~\ref{Festimateproposition} that $\epsilon_0$ is so small that
\eqn{
\|T(w)\|_X&\leq O(\|w\|_X^2)+\frac12\epsilon_1,
}
having also used $t\leq1$. By a similar calculation,
\eqn{
\|T(w_1)-T(w_2)\|_X&\lesssim \|w_1-w_2\|_X(\|w_1\|_X+\|w_2\|_X).
}
Choosing $C_4$ sufficiently large to compensate for the implicit constants, it follows that $T$ is a contraction on the ball $B_X(0,\epsilon_1)$ for all $\epsilon_1\in(0,C_4^{-1})$, and we conclude by the Banach fixed point theorem.

Finally, we prove control of $w$ near the initial time using the Duhamel formula
\eqn{
w(t)=\int_0^te^{(t-t')\Delta}\mathbb P\div (-2v\odot w-w\otimes w+F)(t')dt'.
}
By \eqref{besovheatestimate}, \eqref{vweightedlinftybound}, Proposition~\ref{Festimateproposition}, interpolation, and the fact that $\|w\|_X\leq\epsilon_1$, we have
\eqn{
\|w(t)\|_{\mathcal C^{-1+\alpha/2}}&\lesssim \int_0^t\|-2v\odot w(t')-w\otimes w(t')+F(t')\|_{C^{\alpha/2}}dt'\\
&\lesssim \int_0^t((t')^{-1+\frac\alpha4}+(t')^{-1+\frac{3\alpha}4})dt'\\
&\lesssim t^{\alpha/4}
}
for all $t\in(0,1]$.
\end{proof}

\section{Proof of the Non-Uniqueness Theorem}\label{proofoftheoremsection}

We have shown that for $i\in\{1,2\}$, there exist $v^{(i)}\in C^\infty((0,1]\times\Td)$ and $w^{(i)}\in C^0((0,1];C^{1,\kappa})$ obeying suitable equations such that $u^{(i)}=v^{(i)}+w^{(i)}$ obey \eqref{NSE}. By standard regularity theory, $u^{(i)}$ is in fact smooth on $(0,1]\times\Td$.

Next we argue that both solutions attain the data $U^0$. We showed in Proposition~\ref{wconstructionfixedpointproposition} that $w^{(i)}(t)\to0$ in $\mathcal C^{-1+\alpha/2}(\mathbb{T}^3)$. This subcritical topology is too strong for the principal part of the solutions; instead we claim that $v^{(i)}(t)\to U^0$ in $\dot W^{-1,p}(\mathbb{T}^3)$, for all $1<p<\infty$. Recall from Definitions~\ref{initialdatadefinition}--\ref{principalpartdefinition} that for $i\in\{1,2\}$, we can write
\eqn{
v^{(i)}(t)=\curl\sum_{k\in\mathbb N_i}\curl\psi_k(t)+\sum_{k\in\mathbb N\setminus\mathbb N_i}\overline v_k(t),\quad U^0=\curl\sum_{k\in\mathbb N}\curl\psi_k^0
}
where $\mathbb N_i\coloneqq i-1,i+1,i+3,\ldots$ for $i\in\{1,2\}$. For the heat-dominated part of the data, we use Definition~\ref{principalpartdefinition}, \eqref{abounds}, and \eqref{mikadoprofilebound} to estimate
\eqn{
\Big\|\sum_{k\in\mathbb N_i}(\curl \psi_k(t)-\curl\psi_k^0)\Big\|_{L^{p}}&\lesssim\sum_k\|\curl(a_{j,k}(\Psi_{j,k}-\Psi_{j,k}^0))\|_{L^p}\\
&\lesssim\sum_{j,k}\|\grad(a_{j,k}\Psi_{j,k}^0)\|_{L^p}(1-e^{-|\eta_j|^2\N{k}^2t})\\
&\lesssim\sum_k|\widetilde\Omega_k|^\frac1p(1\wedge(\N{k}^2t))\\
&\leq\sum_{k\,:\,\N{k}\leq t^{-1/4}}\N{k}^2t+\sum_{k\,:\,\N{k}>t^{-1/4}}2^{-\frac kp}.
}
In the last two lines, we used the fact that $\|\nabla(a_{j,k}\Psi_{j,k}^0)\|_{L^\infty}=O(1)$ and that $\supp(a_{j,k}\Psi_{j,k}^0)\subset\widetilde\Omega_k$, whose volume is bounded by $2^{-k}$ as shown in Lemma~\ref{supportlemma}.
The terms in the first sum grow faster than exponentially so the sum is controlled by the largest term which is $O(t^\frac12)$. For the second sum, clearly the first index $k$ satisfying $\N{k+1}>t^{-1/4}$ will grow to infinity as $t\to0$. Since the terms themselves are independent of $t$, the sum is the shrinking tail of a convergent series so it goes to zero. We conclude that $\curl\sum_{k\in\mathbb N_i}\curl\psi_k(t)\to \curl\sum_{k\in\mathbb N}\curl\psi_k^0$ in $\dot W^{-1,p}$.

Likewise for the nonlinear evolution-dominated part, with Definition~\ref{principalpartdefinition}, we compute
\eqn{
&\curl\curl\psi_k^0-\overline{v}_k(t)\\
&\quad=\curl\curl\psi_k^0-\sum_j\mathbb{P}\div(e^{-2|\eta_j|^2\N{k+1}^2t}2^{-1}\N{k+1}^{-2}A_{j,k+1}|\eta_j|^2a_{j,k+1}^2(x)\theta_j\otimes \theta_j)\\
&\quad=\mathbb{P}\div\sum_j(1-e^{-2|\eta_j|^2\N{k+1}^2t})2^{-1}\N{k+1}^{-2}A_{j,k+1}|\eta_j|^2a_{j,k+1}^2(x)\theta_j\otimes \theta_j\\
&\quad\quad+\Big(\curl\curl\psi_k^0-2^{-1}\N{k+1}^{-2}\mathbb{P}\div\sum_ja_{j,k+1}|\eta_j|^2a_{j,k+1}^2(x)\theta_j\otimes \theta_j\Big).
}
By the same computation from Remark~\ref{dataremark}, the quantity in the last line identically vanishes. For the remaining part, we estimate using \eqref{abounds} and Lemma~\ref{supportlemma} that
\eqn{
&\|\mathbb{P}\div\sum_j(1-e^{-2|\eta_j|^2\N{k+1}^2t})2^{-1}\N{k+1}^{-2}A_{j,k+1}|\eta_j|^2a_{j,k+1}^2(x)\theta_j\otimes \theta_j\|_{\dot W^{-1,p}}\\
&\quad\lesssim \sum_j(1-e^{-2|\eta_j|^2\N{k+1}^2t})\N{k+1}^{-2}\|a_{j,k+1}\|_{L^{2p}}^2\\
&\quad\lesssim (1\wedge(\N{k+1}^2t))|\widetilde\Omega_k|^\frac1p
}
for all $p\in(1,\infty)$, since $\|a_{j,k+1}\|_{L^\infty}\lesssim \N{k+1}$ and $\supp a_{j,k+1}\subset\widetilde\Omega_k$ by Definitions~\ref{Omegadefintion} and \ref{dataiteratedefinition}. Summing, we find for all $t<1$ that $\sum_{k\in\mathbb N\setminus\mathbb N_i}(\overline v_k(t)-\curl\curl\psi^0_k)$ is absolutely summable in the $\dot W^{-1,p}$ norm and obeys the bound
\eqn{
\|\sum_{k\in\mathbb N\setminus\mathbb N_i}(\overline v_k(t)-\curl\curl\psi^0_k)\|_{\dot W^{-1,p}}&\lesssim\sum_k|\widetilde\Omega_k|^\frac1p(1\wedge(\N{k+1}^2t))
}
which converges to zero by the same reasoning as in the previous case. This completes the proof that the initial data is attained by both solutions.

Finally, we show that the two solutions are distinct. We consider a time scale when all but the lowest\footnote{This choice is made for simplicity, but any mode would exhibit the difference between the solutions on the appropriate time scale.} frequency mode should have dissipated away, for instance $t_0=\N{0}^{-2}$. Recall from Proposition~\ref{wconstructionfixedpointproposition} that $w\in B_X(0,\epsilon_1)$. From the triangle inequality,
\eqn{
\|v^{(1)}(t_0)-v^{(2)}(t_0)\|_{L^\infty}&\geq\|v_0(t_0)\|_{L^\infty}-\sum_{k\geq1}\|v_k(t_0)\|_{L^\infty}-\sum_{k\geq0}\|\overline v_k(t_0)\|_{L^\infty}\\
&\quad-\|w^{(1)}(t_0)\|_{L^\infty}-\|w^{(2)}(t_0)\|_{L^\infty}.
}
We estimate using \eqref{vkbounds} that
\eqn{
\sum_{k\geq1}\|v_k(t_0)\|_{L^\infty}&\lesssim\sum_{k\geq1}\N{k}\exp(-\N{k}^2\N{0}^{-2})\lesssim \N{1}\exp(-\N{1}^2\N{0}^{-2})\\
&\leq \N{0}(\N{1}/\N{0})^{-100}
}
and
\eqn{
\sum_{k\geq0}\|\overline v_k(t_0)\|_{L^\infty}&\lesssim\sum_{k\geq0}\N{k}\exp(-\N{k+1}^2\N{0}^{-2})\lesssim \N{0}\exp(-\N{1}^2\N{0}^{-2})
}
which is even smaller than the previous bound. For the $w^{(i)}$, we have by Proposition~\ref{wconstructionfixedpointproposition}
\eqn{
\|w^{(i)}(t_0)\|_{L^\infty}&\lesssim \epsilon_1t_0^{-\frac12+\frac\alpha2}=\epsilon_1\N{0}^{1-\alpha}=\epsilon_1A^{-\alpha\gamma}\N{0}.
}

It remains only to prove a lower bound on $v_0(t_0)$. Recall the notation of the proof of Proposition~\ref{Festimateproposition}, in particular the decomposition $v_0=v_0^p+v_0^e$. By \eqref{vkebounds},
\eqn{
\|v_0^e(t_0)\|_{L^\infty}\lesssim \M0\leq2\N{0}A^{-\gamma+1}.
}
We prove the lower bound on $v_0^p(x,t)=\N{0}^3|\eta_1|^2\Psi_{1,0}(x,t)$, recalling from Definition~\ref{dataiteratedefinition} that $a_{j,0}$ takes a particularly simple form when $k=0$. By inspection of Definition~\ref{mikadodefinition}, we easily have $\|\Psi_{1,0}(t_0)\|_{L^\infty}\gtrsim \N{0}^{-2}$
from which it follows $\|v_0^p(t_0)\|_{L^\infty}\geq C_5^{-1}\N{0}$ for a large enough constant $C_5$. Collecting everything,
\eqn{
\|v^{(1)}(t_0)-v^{(2)}(t_0)\|_{L^\infty}\geq C_5^{-1}\N{0}-A^{-c(b,\alpha,\gamma)}\N{0}
}
for a $c(b,\alpha,\gamma)>0$. With a sufficiently large choice of $A$, we conclude that $v^{(1)}$ and $v^{(2)}$ are distinct.

Finally, let us extend our solutions on $[0,1]$ to be global solutions. Recall from \eqref{vweightedlinftybound} that $\|\grad^m v^{(i)}|_{t=1}\|_{L^\infty}\leq \exp(-\N{0}^2/O_m(1))$ for all $m\geq0$. It follows that $\|P_Nv^{(i)}|_{t=1}\|_{L^\infty}\lesssim \langle N\rangle^{-5}\exp(-A^{2\gamma})$ and therefore $v^{(i)}|_{t=1}$ can be made arbitrarily small in, say, $H^\frac12(\Td)$ or $B^{-1}_{\infty,2}\subset BMO^{-1}$, and existence of a global-in-time solution follows. \null\hfill\qedsymbol

\appendix

\section{Tools from Convex Integration}\label{CIappendix}

The construction of the initial data $U^0$ in \S\ref{principalpartsection} makes use of several tools appearing in works on convex integration. In this appendix we quote those results for the reader's convenience.

First, we have the following rank-1 decomposition of a symmetric tensor field. This idea appeared first in the work of Nash~\cite{nash1954c}, and subsequently in many of the modern works on fluid equations.

\begin{lem}[Nash lemma]\label{nash-lemma}
There exist $c_0>0$, $\theta_1,\ldots,\theta_6\in\Zd$, and $\Gamma_1,\ldots,\Gamma_6\in C^\infty(B;\mathbb R)$ such that
\eqn{
M=\sum_{j=1}^6\Gamma_j(M)^2\theta_j\otimes \theta_j
}
for all $M\in B$, where $B=B(\Id,c_0)$ is the closed ball of radius $c_0$ centered at the identity in $S^{3\times 3}$. The $\Gamma_j$ obey the estimates
\eq{\label{Gammabound}
\frac1{100}\leq\Gamma_j\leq1\quad\forall M\in B
}
and
\eq{\label{GammaboundII}
\|\grad^m\Gamma_j\|_{L^\infty(B)}\lesssim_m1
}
for all $m\geq1$.
\end{lem}

\begin{proof}
Let us define the vectors
\eqn{
\theta_1=(0,0,1),\quad \theta_2=(2,0,1),\quad \theta_3=(1,1,1),\\ \theta_4=(-1,1,1),\quad \theta_5=(-2,0,1),\quad \theta_6=(0,-2,1).
}
For any $M\in B$, let $\epsilon=M-\Id$. One can verify that the claim holds if
\eqn{
\Gamma_1(M)^2&=\frac{1}{3}-\frac{\epsilon _{11}}{4}-\frac{5 \epsilon _{22}}{12}-\frac{\epsilon _{23}}{3}+\epsilon _{33}\\
\Gamma_2(M)^2&=\frac{1}{12}+\frac{\epsilon _{11}}{8}-\frac{\epsilon _{12}}{4}+\frac{\epsilon _{13}}{4}-\frac{\epsilon _{22}}{24}-\frac{\epsilon _{23}}{12}\\
\Gamma_3(M)^2&=\frac{1}{6}+\frac{\epsilon _{12}}{2}+\frac{\epsilon _{22}}{6}+\frac{\epsilon _{23}}{3}\\
\Gamma_4(M)^2&=\frac{1}{6}-\frac{\epsilon _{12}}{2}+\frac{\epsilon _{22}}{6}+\frac{\epsilon _{23}}{3}\\
\Gamma_5(M)^2&=\frac{1}{12}+\frac{\epsilon _{11}}{8}+\frac{\epsilon _{12}}{4}-\frac{\epsilon _{13}}{4}-\frac{\epsilon _{22}}{24}-\frac{\epsilon _{23}}{12}\\
\Gamma_6(M)^2&=\frac{1}{6}+\frac{\epsilon _{22}}{6}-\frac{\epsilon _{23}}{6}.
}
By taking $c_0=\frac1{1000}$ (say), one can arrange that each of these expressions is bigger than $1/100$. Thus there is a smooth positive choice of $(\Gamma_j)_{j=1,\ldots,6}$. The estimates \eqref{Gammabound}--\eqref{GammaboundII} are elementary.
\end{proof}

Next, we define the following anti-divergence operator which first appeared in \cite{de2013dissipative}.

\begin{define}\label{Rdefinition}
We define the operator $\mathcal R:C^\infty(\Td;\mathbb R^d)\to C^\infty(\Td;S^{3\times 3})$ by $(\mathcal RV)_{ij}\coloneq\mathcal R_{ijk}V_k$ where
\eqn{
\mathcal R_{ijk}=-\frac12\Delta^{-2}\dd^3_{ijk}-\frac1{2}\Delta^{-1}\delta_{ij}\dd_k+\Delta^{-1}\delta_{jk}\dd_i+\Delta^{-1}\delta_{ik}\dd_j
}
where $\delta$ denotes the Kronecker delta.
\end{define}

The utility of Definition~\ref{Rdefinition} is that $\mathcal R$ is a right inverse for the divergence operator acting on symmetric tensor fields. For a proof, see for instance Lemma 4.3 in \cite{de2013dissipative}.

\begin{lem}\label{Rlemma}
For any $V\in C^\infty(\Td;\mathbb R^d)$, we have $\mathcal RV\in C^\infty(\Td;S^{3\times3}(\mathbb R))$ as well as the identity
\eqn{
    \div\mathcal RV=V-\fint_{\Td} V.
}
\end{lem}

Variants of the following stationary phase lemma can be found in the papers of Buckmaster--De Lellis--Isett--Sz\'ekelyhidi~\cite{buckmaster2015anomalous} and De Lellis--Sz\'ekelyhidi~\cite{de2013dissipative}.

\begin{lem}\label{stationaryphase}
With $\mathcal R$ as in Definition~\ref{Rdefinition},
\eqn{
\|\mathcal R(a(x)e^{i\lambda k\cdot x})\|_{C^\beta}&\lesssim_{\beta,m}\lambda^{-1+\beta}\|a\|_{L^\infty}+\lambda^{-m+\beta}\|\grad^ma\|_{L^\infty}+\lambda^{-m}\|\grad^ma\|_{C^\beta}.
}

\end{lem}

Finally, we quote the following ``improved H\"older's inequality'' that first appeared in work of Modena--Sz\'ekelyhidi \cite[Lemma 2.1]{modena2018non}, inspired by a similar lemma of Buckmaster--Vicol~\cite{buckmaster2019nonuniqueness}. The idea, roughly speaking, is that the gains due to intermittency at separated scales is multiplicative.

\begin{lem}\label{BVlemma}
For $f,g\in C^\infty(\Td)$ and $\lambda\in\mathbb N$,
\eqn{
\|f(\cdot)g(\lambda\cdot)\|_{L^1}\leq\|f\|_{L^1}\|g\|_{L^1}+O(\lambda^{-1}\|f\|_{C^1}\|g\|_{L^1}).
}
\end{lem}

\section{Tools for Constructing the Semigroup}

\subsection{Existence for the Perturbed Stokes System}\label{linearappendix}

In the interest of the argument being self-contained, we construct solutions to the Navier--Stokes linearized around the principal part of the solutions. This will be necessary for defining the semigroup $S(t;t')$ in \S\ref{semigroupsubsection}. One can find a much more precise result in the book of Lemarie-Rieusset~\cite[Theorem 21.3]{lemarie2002recent} where the drift is allowed to be in a very refined critical space.

\begin{proposition}\label{linearweaksolutionsexistence}
Fix $t_0\in \mathbb{R}$.
Let $v\in C^\infty([t_0,\infty)\times\Td;\Rd)$ and $w^0\in C^\infty(\Td;\Rd)$ be divergence-free vector fields. Then there exists a solution $w\in C^\infty([t_0,\infty)\times\Td;\Rd)$ of
\begin{equation}\begin{aligned}\label{transportdiffusionequation}
    \dd_t w-\Delta w+v\cdot\grad w+w\cdot\grad v+\grad p=0\\
    \div w=0\\
    w(t_0,x)=w^0(x).
\end{aligned}\end{equation}
\end{proposition}

\begin{proof}
It suffices to prove the proposition on the finite interval $[t_0,T]$ for all $T>0$. We begin by proving the existence of a weak solution in $L_t^\infty L_x^2\cap L_t^2H_x^1([t_0,T]\times\Td;\Rd)$ by the Galerkin method as in \cite[Theorem 3.5]{tsai2018lectures}. Fix $\{\phi_j\}_{j=1}^\infty\subset C^\infty(\Td;\Rd)$, a divergence-free orthonormal basis of $H^1(\mathbb{T}^d; \mathbb{R}^d)$. Consider the finite rank approximation of $w_0$ in this basis:
$$ w^0_m(x) = \sum_{j=1}^{m} \phi_j(x)(\phi_j, w^0)_{L^2}.$$
We search for a solution of \eqref{transportdiffusionequation} in this basis, namely a vector field of the form
\eq{\label{finitedimensionalansatz}w_m(x,t) = \sum_{j=1}^m g_j(t)\phi_j(x)}
and desire to solve a system of ODE's for the coefficient functions $g_j(t)$ with initial data given by $g_j(t_0)= (\phi_j, w^0)$.

We would like to choose the coefficients in the ansatz \eqref{finitedimensionalansatz} so that $w_m$ exactly obeys \eqref{transportdiffusionequation}. Of course this is not possible because the partial basis is not closed under the operations in the equation; instead, we solve \eqref{transportdiffusionequation} projected onto the subspace. Namely,
\eqn{
(\phi_i,\sum_{j=1}^m(g_j'\phi_j-g_j\Delta\phi_j+g_jv\cdot\grad\phi_j+g_j\phi_j\cdot\grad v))_{L^2}=0
}
which leads to the ODE system
\eqn{g_i'(t) +a_{ij}(t)g_j =0\\
g_i|_{t=t_0}=(\phi_i,w^0)_{L^2}}
for $i=1,2,\ldots,m$, where
\eqn{a_{ij}(t)=(\varphi_i,-\Delta\varphi_j+v(t)\cdot\grad\varphi_j+\varphi_j\cdot\grad v(t))_{L^2}.}

Because $v\in L^1([t_0,\infty);C^1(\Td))$, the norm of $a_{ij}$ is integrable and standard ODE theory gives a unique global solution $g_1(t), \ldots g_m(t)$ to the system. Having constructed the coefficients, we can build the function $w_m$ as defined in \eqref{finitedimensionalansatz}. By construction, $w_m$ obeys \eqref{transportdiffusionequation} projected onto any of the $\phi_1,\ldots,\phi_m$. In particular we may test against $w_m$ to obtain
$$\frac{1}{2}\int_{\mathbb{T}^d}|w_m(t)|^2 dx + \int_{t_0}^t \int_{\mathbb{T}^d} |\nabla w_m|^2 dxdt = \frac{1}{2}\int_{\mathbb{T}^d} |w_m^0|^2 dx +\int_{t_0}^t\int_{\mathbb{T}^d} (w_m \cdot \nabla w_m) \cdot v dxdt,$$
having used that $v$ and $w_m$ are divergence-free. We estimate the last term using the Peter-Paul inequality:
\eqn{
\int_{t_0}^t\int_{\mathbb{T}^d} (w_m \cdot \nabla w_m) \cdot v dxdt &\leq \int_{t_0}^t \|w_m \|_{L^2}\|\nabla w_m\|_{L^2}\|v\|_{L^\infty}dt\\
&\leq \int_{t_0}^t \Big(\frac14\|\nabla w_m\|_{L^2}^2 + 4\|w_m\|_{L^2}^2\|v\|_{L^\infty}^2\Big)dt.
}

Inserting this back in to the energy equality, we conclude that
$$\frac{1}{2}\int_{\mathbb{T}^d}|w_m(t)|^2 dx  +\frac{1}{4}\int_{t_0}^t \int_{\mathbb{T}^d} |\nabla w_m|^2dxdt\leq \frac{1}{2}\int_{\mathbb{T}^d} |w_m^0|^2 dx +\int_{t_0}^t 4\|w_m\|_{L^2}^2\|v\|_{L^\infty}^2dt.$$
By Gr\"onwall's inequality, we have
$$\|w_m\|_{L^\infty_t L^2_x([t_0,T] \times \mathbb{T}^d)}^2 +\frac{1}{2}\|\nabla w_m \|_{L^2_t L^2_x([t_0,T] \times \mathbb{T}^d)}^2 \leq \|w^0\|_{L^2(\Td)}^2 + 8 \exp( \|v\|_{L^2_tL^\infty_x([t_0,T] \times \mathbb{T}^d)}^2).$$
Thus $w_m$ is a priori bounded in the energy space, uniformly in $m$. As a result, there exists a limiting function $w_\infty(x,t)$ such that $w_m$ converges weak-$\ast$ to $w_\infty$ in $L^\infty_t L^2_x$ and $w_m$ converges weakly to $w_\infty$ in $L^2_t L^2_x$. Moreover $w_\infty$ satisfies the same a priori bound as the $w_m$. To show that $w_\infty$ is a weak solution of \eqref{transportdiffusionequation}, it suffices to test against functions of the form $\zeta(x,t):=\theta(t) \phi_j(x)$. By construction, for all $m\geq j$ we have
$$\int_{t_0}^T\int_{\mathbb{T}^d} -w_m\cdot \partial_t \zeta -  w_m \cdot \Delta \zeta -(v\odot w_m) :\nabla \zeta dx =0.
$$
By the weak convergence of $w_m$, we can pass to the limit and conclude $w_\infty$ satisfies \eqref{transportdiffusionequation}. From the energy bounds on $w=w_\infty$, it is easy to see that $\partial_t w_\infty \in L^2_t H^{-1}_x$ exists in the weak sense, so we in fact have $w_\infty \in C([t_0,T], L^2_x)$ and convergence to the initial data $w^0$.

Finally, one can upgrade this weak solution to a strong one. One approach is to consider $w$ as a solution of the forced Stokes equation
\eqn{
\dd_tw-\Delta w+\grad p=\div f\\
\div w=0\\
w(t_0,x)=w^0(x)
}
where $f=2v\odot w\in L_t^\infty L_x^2$. The weak solution can be expressed in terms of the Stokes semigroup $e^{t\Delta}\mathbb P$ as
\eqn{
w(t)=e^{t\Delta}w^0+\int_0^te^{(t-t')\Delta}\mathbb P\div f(t')dt'.
}
From this formula, one can iteratively bootstrap regularity using the smoothing effect of the heat kernel via \eqref{heatandderivativemultiplierestimate}. For instance,
\eqn{
\|w(t)\|_{H^\frac12}&\lesssim \|w^0\|_{H^\frac12}+\int_0^t(t-t')^{-\frac34}\|f(t')\|_{L^2}dt'<\infty,
}
and so on.
\end{proof}

\subsection{Gr\"onwall inequalities}

The following nonlinear Gr\"onwall inequality will be used to prove the fractional Gr\"onwall inequality (Lemma~\ref{gronwalllemma}). It is similar to some that have appeared before (e.g. \cite{willett1964nonlinear}) but we could not locate the version necessary for our purposes in the literature.

\begin{lem}\label{nonlineargronwall}
Let $f(t)$, $h_i(t)$, $b_i(t)$, and $a(t)$ be non-negative continuous functions on $[{t_0},\infty)$, with $b_i(t)$ and $a(t)$ non-decreasing. Suppose for some $p>1$,
\eqn{
f(t)\leq a(t)+b_1(t)\int_{t_0}^th_1(s)f(s)ds+(b_2(t)\int_{t_0}^th_2(s)f(s)^pds)^\frac1p.
}
Then for any constant $C_6$ large enough so that $C_6^{-1}+(C_6p)^{-\frac1p}<1$, we have
\eq{\label{nonlineargronwallclaim}
f(t)\leq C_6a(t)\exp\Big(C_6(b_1(t)\int_{t_0}^th_1(s)ds+b_2(t)\int_{t_0}^th_2(s)ds)\Big).
}
\end{lem}

\begin{proof}
First, assume $a$ and $b_i$ are constant. In this case the $b_i$ can be absorbed into the $h_i$. The proof is by a continuity argument. Consider the set $I\subset[{t_0},\infty)$ defined by
\eqn{
I=\big\{T\in[{t_0},\infty):\text{ \eqref{nonlineargronwallclaim} for all }t\in[{t_0},T]\big\}.
}
Then $I$ is clearly closed (by continuity) and non-empty (because of the trivial point $t={t_0}$). The lemma is proved once we show that $I$ is open in $[{t_0},\infty)$. Suppose we have $T\in I$. Then
\eqn{
f(T)&\leq a+\int_{t_0}^Th_1(s)f(s)ds+(\int_{t_0}^Th_2(s)f(s)^pds)^\frac1p\\
&\leq a+C_6a\int_{t_0}^Th_1(s)\exp(C_6\int_{t_0}^s(h_1+h_2))ds\\
&\quad+C_6a(\int_{t_0}^Th_2(s)\exp(pC_6\int_{t_0}^s(h_1+h_2))ds)^\frac1p\\
&\leq a+a\int_{t_0}^T\frac d{ds}\exp(C_6\int_{t_0}^s(h_1+h_2))ds\\
&\quad+(C_6p)^{-\frac1p}C_6a(\int_{t_0}^T\frac d{ds}\exp(pC_6\int_{t_0}^s(h_1+h_2))ds)^\frac1p\\
&=a\exp(C_6\int_{t_0}^T(h_1+h_2))+(C_6p)^{-\frac1p}C_6a(\exp(pC_6\int_{t_0}^T(h_1+h_2)ds)-1)^\frac1p\\
&\leq a(1+(C_6p)^{-\frac1p}C_6)\exp(C_6\int_{t_0}^T(h_1+h_2)).
}
By assumption, the coefficient is strictly less than $C_6$; thus by continuity, there is an $\epsilon>0$ such that $[T,T+\epsilon)\subset I$. We conclude that $I=[{t_0},\infty)$.

Now we consider the case of non-decreasing $a$ and $b_i$. Fix any $T>0$. We have
\eqn{
f(t)\leq a(T)+b_1(T)\int_{t_0}^th_1(s)f(s)ds+(b_2(T)\int_{t_0}^th_2(s)f(s)^pds)^\frac1p
}
for all $t\in[{t_0},T]$. We may apply the ``constant'' version of the lemma proved above to conclude 
\eqn{
f(t)\leq C_6a(T)\exp\Big(C_6(b_1(T)\int_{t_0}^th_1(s)ds+b_2(T)\int_{t_0}^th_2(s)ds)\Big)\quad \forall t\in[{t_0},T].
}
In particular, this holds at $t=T$ which proves the general claim.
\end{proof}

Gr\"onwall-type inequalities similar to the one below have appeared in the literature and are known as ``fractional Gr\"onwall'' inequalities (e.g., \cite{webb2021fractional}). The one required in \S\ref{semigroupsubsection} has an endpoint character. One can see, for instance, that the extra factor $\|s^\frac12g_2\|_{L^\infty}$ in the exponential reflects the failure of the Hardy--Littlewood--Sobolev inequality at the $L^\infty$ endpoint.

\begin{lem}\label{gronwalllemma}
If $a(t)$, $f(t)$, $g_1(t)$, and $g_2(t)$ are positive continuous functions on $[t_0,\infty)$ with $a(t)$ non-decreasing and
\eq{\label{gronwallhypothesis}
f(t)\leq a(t)+\int_{t_0}^{t}(g_1(s)+(t-s)^{-\frac12}g_2(s))f(s)ds\quad\forall t\geq t_0
}
then for any $p>2$,
\eqn{
f(t)\lesssim_p a(t)\exp\Big(O_p(\int_{t_0}^tg_1(s)ds+\|s^\frac12 g_2\|_{L^\infty([t_0,t])}^{p-2}\int_{t_0}^tg_2(s)^2ds)\Big)\quad \forall t\geq t_0.
}
\end{lem}

\begin{proof}
By H\"older's inequality,
\eqn{
\int_{t_0}^t(t-s)^{-\frac12}g_2(s)f(s)ds&\leq (\int_{t_0}^t(t-s)^{-\frac{p'}2}g_2(s)^{(1-\frac2p)p'}ds)^{\frac1{p'}}(\int_{t_0}^tg_2(s)^2f(s)^pds)^{\frac1p}.
}
The first factor on the right-hand side is controlled by
\eqn{
\|s^\frac12g_2\|_{L^\infty}^{1-\frac2p}(\int_{t_0}^t(t-s)^{-\alpha}s^{-\beta}ds)^{1-\frac1p}
}
where $\alpha=\frac p{2(p-1)}$ and $\beta=\frac{p-2}{2(p-1)}$. In general we have that when $\alpha\in(0,1)$ and $\alpha+\beta=1$ then
\eqn{
\int_{t_0}^t(t-s)^{-\alpha}s^{-\beta}ds\leq\frac\pi{\sin\pi\alpha}\quad\forall t\geq t_0
}
with equality when $t_0=0$. In summary,
\eqn{
f(t)\leq a(t)+\int_{t_0}^tg_1(s)f(s)ds+O_p(\|s^\frac12 g_2\|_{L^\infty([t_0,t])}^{1-\frac2p}(\int_{t_0}^tg_2(s)^2f(s)^pds)^\frac1p)
}
so we conclude by Lemma~\ref{nonlineargronwall}, using that the coefficient function $t\mapsto \|s^\frac12 g_2\|_{L^\infty([t_0,t])}^{1-\frac2p}$ is non-decreasing.
\end{proof}

\section{Multiplier Estimates}\label{multiplierappendix}

Proving the basic properties of Littlewood--Paley calculus is slightly more subtle on the torus than in the whole space due to the lack of scaling symmetry, so we include the proofs in the interest of being self-contained. The idea is that one can move between $\Td$ and $\Rd$ using the Poisson summation formula; see for instance the book of Schmeisser and Triebel \cite{SchmeisserTriebel1987}.

\subsection{A general estimate}

Let $a:\Zd\to\mathbb C$. We reserve the notation $\mathcal F$ and $\mathcal F^{-1}$ for the forward and inverse Fourier transforms on $\Td$; see \S\ref{notationsection}. We also recall the forward and inverse Fourier transforms on $\Rd$: for Schwartz functions $\varphi\in\mathcal S(\Rd;\mathcal C)$, we define
\eqn{
\mathcal F_{\Rd}\varphi(x)\coloneqq\int_\Rd\varphi(\xi)e^{-ix\cdot\xi}d\xi,\quad\mathcal F^{-1}_{\Rd}\varphi(x)\coloneqq(2\pi)^{-3}\int_\Rd\varphi(\xi)e^{ix\cdot\xi}d\xi.
}
Let us briefly recall a version of the Poisson summation formula. Let $f\in\mathcal S(\Rd;\mathbb C)$ be Schwartz and define the periodization
\eqn{
f_p:\Td\to\mathbb C;\quad f_p(x)\coloneqq\sum_{n\in(2\pi\mathbb Z)^3}f(x+n).
}
Then it is straightforward to compute the following relation between the Fourier series of $f_p$ and the Fourier transform of $f$:
\eq{\label{poissonsummationstartingpoint}
\hat f_p(\xi)=(2\pi)^{-3}\mathcal F_{\Rd}f(\xi)\quad\forall\xi\in\Zd.
}
By setting $\xi=0$ in \eqref{poissonsummationstartingpoint}, one obtains the standard version of the Poisson summation formula. More useful for us is the following observation: let $f=\mathcal F^{-1}_{\Rd}\varphi$ with $\varphi\in\mathcal S(\Rd;\mathbb C)$. Then \eqref{poissonsummationstartingpoint} implies that $\hat f_p(\xi)=(2\pi)^{-3}\varphi(\xi)$ for all $\xi\in\Zd$. Viewing $\varphi$ as a function $\Zd\to\mathbb C$, it follows that $\mathcal F^{-1}\varphi=(2\pi)^3f_p$. In other words,
\eq{\label{poissonsummationformula}
\mathcal F^{-1}\varphi(x)=(2\pi)^3\sum_{\xi\in\Zd}\mathcal F^{-1}_\Rd\varphi(x+2\pi \xi)\quad\forall x\in\Td.
}
Having established \eqref{poissonsummationformula}, we can proceed to prove the multiplier bounds.

\begin{lem}[General multiplier boundedness]\label{generalmultiplierlemma}
If $m\in \mathcal S(\Rd;\mathbb C)$, $1\leq p\leq q\leq\infty$, and $n\geq 4$, then
\eqn{
\|m(\grad/i)f\|_{L^q(\Td)}&\lesssim_{p,q,n}C_n(m)\|f\|_{L^p(\Td)}
}
for all $f\in L^p(\Td;\mathbb R)$ where
\eqn{
C_n(m)&=\|\mathcal F_{\Rd}^{-1}m\|_{L^1(\Rd)}^{1+1/q-1/p}\max(\|m\|_{L^1(\Rd)},\|m\|_{L^1(\Rd)}^{1-3/n}\|\grad^nm\|_{L^1(\Rd)}^{3/n})^{1/p-1/q}.
}
\end{lem}

\begin{proof}
By density, we may assume without loss of generality that $f\in C^\infty(\Td;\mathbb R)$.

The operator $m(\grad/i)$ can be expressed as convolution on $\Td$ with the kernel
\eqn{
K(x)=\mathcal F^{-1}m(x)=(2\pi)^3\sum_{k\in\Zd}\mathcal F^{-1}_{\Rd}m(x+2\pi k)
}
using \eqref{poissonsummationformula}. Recalling that we identify $\Td$ with $[0,2\pi]^3$,
\eqn{
\|K\|_{L^1(\Td)}&\leq(2\pi)^3\sum_{k\in\Zd}\|\mathcal F_{\Rd}^{-1}m(x+2\pi k)\|_{L^1(\Td)}=(2\pi)^3\|\mathcal F_{\Rd}^{-1}m\|_{L^1(\Rd)}.
}
Moreover, by the Riemann--Lebesgue lemma,
\eqn{
|K(x)|&\leq(2\pi)^3\sum_{k\in\Zd}|\mathcal F_{\Rd}^{-1}m(x+2\pi k)|\\
&\lesssim\sum_{k\in\Zd}\min(\|m\|_{L^1(\Rd)},|x+2\pi k|^{-n}\|\grad^nm\|_{L^1(\Rd)}).
}
If $\|\grad^nm\|_{L^1(\Rd)}> \|m\|_{L^1(\Rd)}$, then one splits the sum based on which item in the minimum dominates to obtain $\|m\|_{L^1(\Rd)}^{1-3/n}\|\grad^nm\|_{L^1(\Rd)}^{3/n}$. Otherwise, one can estimate at $|k|\leq10$ using $\|m\|_{L^1}$ and at $|k|>10$ with the other bound to obtain $\|m\|_{L^1}+\|\grad^nm\|_{L^1}\lesssim \|m\|_{L^1}$.

The lemma follows by Young's inequality and interpolation,
\eqn{
\|K\|_{L^r}\leq\|K\|_{L^1}^{1/r}\|K\|_{L^\infty}^{1-1/r}
}
where $\frac1r=1+\frac1q-\frac1p$.
\end{proof}
We are now equipped to prove the Lemmas~\ref{bernsteininequality}--\ref{heatlplemma}.

\subsection{Proof of Lemma~\ref{bernsteinlemma}}

Recall from \S\ref{littlewoodpaleysection} the symbol $\pi(\xi/N)$ for the Fourier multiplier $P_N$. Then we have
\eqn{
\|\mathcal F^{-1}_\Rd(\pi(\cdot/N)m)\|_{L^1(\Rd)}&=N^d\|(\mathcal F^{-1}_\Rd(m(N\cdot)\pi))(N\cdot)\|_{L^1(\Rd)}=\|\mathcal F^{-1}_{\Rd}(m(N\cdot)\pi)\|_{L^1(\Rd)}
}
by homogeneity. Recall that by definition, $\pi(\cdot/N)$ is supported in an annulus at spatial scale $\sim N$. By the Riemann--Lebesgue lemma, H\"older, and \eqref{multiplierassumption}, we have both the $L^\infty$ bound
\eqn{
\|\mathcal F^{-1}_{\Rd}(m(N\cdot)\pi)\|_{L^\infty(\Rd)}&\lesssim \|m(N\cdot)\pi\|_{L^1}\lesssim AN^{\alpha}
}
and the weighted $L^\infty$ bound
\eqn{
\||x|^{10}\mathcal F^{-1}_{\Rd}(m(N\cdot)\pi)\|_{L^\infty(\Rd)}&\lesssim \|\grad^{10}(m(N\cdot)\pi)\|_{L^1}\lesssim AN^{\alpha}.
}
One may easily combine these pointwise estimates to control the $L^1$ norm and obtain
\eqn{\|\mathcal F^{-1}_\Rd(\pi(\cdot/N)m)\|_{L^1(\Rd)}\lesssim AN^\alpha.}
Finally, one estimates
\eqn{
\|\grad^n(\pi(\cdot/N)m)\|_{L^1(\Rd)}&=AN^{\alpha-n}\|1\|_{L^1(\supp\pi(\cdot/N))}=O(AN^{3+\alpha-n})
}
and the claim follows by Lemma~\ref{generalmultiplierlemma}.\null\hfill\qedsymbol

\subsection{Proof of Lemma~\ref{heatlplemma}}

With $m$ the multiplier for $P_Ne^{t\Delta}$, it is straightforward to obtain the bound
\eqn{
|\mathcal F^{-1}_{\Rd}m(x)|&\lesssim \min(1,\Big(\frac{1+N^2t}{|x|}\Big)^{10})\exp(-N^2t/O(1))
}
which implies $\|\mathcal F^{-1}_{\Rd}m\|_{L^1(\Rd)}\lesssim\exp(-N^2t/O(1))$. Then \eqref{heatlittlewoodpaleymultiplierestimate} follows by Lemma~\ref{generalmultiplierlemma}.

Let us turn to \eqref{stokesequationsolutionoperatorestimate}, first assuming that $t\leq N^{-2}$. Define the multiplier
\eqn{
m(\xi)=e^{-t|\xi|^2}\pi_N(\xi)b(\xi).
}
We easily have
\eqn{
|\grad^nm(\xi)|\lesssim |\xi|^\alpha t^{\frac n2}(t^\frac12|\xi|)^{-n}\mathbf1_{|\xi|\sim N}
}
for all $n\geq0$ which leads to
\eqn{
\|\grad ^nm(\xi)\|_{L^1(\Rd)}&\lesssim N^{3+\alpha-n}
}
and thus
\eqn{
\|\mathcal F^{-1}_\Rd m\|_{L^1(\Rd)}&\lesssim N^\alpha
}
for this set of times. On the other hand, when $t>N^{-2}$,
\eqn{
|\grad ^nm(\xi)|&\lesssim |\xi|^\alpha t^{\frac n2}\exp(-N^2t/O(1))\mathbf1_{|\xi|\sim N}.
}
Note that we can remove all factors of $(t^\frac12|\xi|)$ by enlarging the $O(1)$ in the exponential, as needed. It follows that
\eqn{
\|\mathcal F^{-1}m\|_{L^1(\Rd)}&\lesssim t^{\frac32}N^{3+\alpha}\exp(-N^2t/O(1))\lesssim N^\alpha\exp(-N^2t/O(1)).
}
Thus we conclude the desired bounds for all times by Lemma~\ref{generalmultiplierlemma}.

Next, if $\hat f(0)=0$ and $t>1$, we can decompose $f=\sum_{N\geq1}P_Nf$ and directly obtain \eqref{heatmultiplierestimate} from \eqref{heatlittlewoodpaleymultiplierestimate}. If instead $t\leq1$, that approach would not lead to the optimal estimate because there is a $O(1)$ contribution for each of the modes between $N=1$ and $N\sim t^{-\frac12}$, leading to a logarithmic loss. Instead we decompose
\eqn{
\|e^{t\Delta} f\|_{L^\infty(\Td)}&\lesssim \|P_{1\leq\cdot\leq t^{-\frac12}}e^{t\Delta}f\|_{L^\infty}+\sum_{N>t^{-\frac12}}\|P_Ne^{t\Delta}f\|_{L^\infty}.
}
The sum over large $N$ can be controlled by $O(1)$ using \eqref{heatlittlewoodpaleymultiplierestimate}. Then one can see the operator $P_{1\leq\cdot\leq t^{-1/2}}e^{t\Delta}$ is uniformly bounded on $L^\infty$ by appealing to Lemma~\ref{generalmultiplierlemma}.

Finally, \eqref{heatandderivativemultiplierestimate} follows from \eqref{heatlittlewoodpaleymultiplierestimate} in a similar (more straightforward) way. Because $\alpha>0$, the low modes are summable even when $t$ is small.\null\hfill\qedsymbol

\bibliographystyle{abbrv}
\bibliography{references}

\end{document}